\documentclass[11pt]{article}
\usepackage{amsthm,amsmath,amssymb}
\usepackage[numbers]{natbib}
\usepackage{multirow}
\usepackage{caption}
\usepackage{subcaption}
\usepackage{float}
\usepackage{makecell}
\usepackage{booktabs}
\usepackage{array}
\usepackage{fullpage}
\usepackage{url}
\usepackage{algorithm}
\usepackage{algorithmic}
\usepackage{bm}
\usepackage{authblk}
\usepackage{cancel}
\usepackage{bbm}
\usepackage{smile}
\usepackage{mathtools}
\usepackage{lipsum}
\usepackage{mathrsfs}
\usepackage{dsfont}
\usepackage{titling}
\usepackage{epstopdf}
\usepackage{caption}
\usepackage{marvosym}
\usepackage{diagbox}
\usepackage{slashbox}
\usepackage{fullpage}
\usepackage{notoccite}

\usepackage{epsfig}
\usepackage{pst-grad} 
\usepackage{pst-plot} 
\usepackage[space]{grffile} 
\usepackage{etoolbox} 
\makeatletter 
\patchcmd\Gread@eps{\@inputcheck#1 }{\@inputcheck"#1"\relax}{}{}
\makeatother

\makeatletter
\DeclareRobustCommand\widecheck[1]{{\mathpalette\@widecheck{#1}}}
\def\@widecheck#1#2{%
    \setbox\z@\hbox{\m@th$#1#2$}%
    \setbox\tw@\hbox{\m@th$#1%
       \widehat{%
          \vrule\@width\z@\@height\ht\z@
          \vrule\@height\z@\@width\wd\z@}$}%
    \dp\tw@-\ht\z@
    \@tempdima\ht\z@ \advance\@tempdima2\ht\tw@ \divide\@tempdima\thr@@
    \setbox\tw@\hbox{%
       \raise\@tempdima\hbox{\scalebox{1}[-1]{\lower\@tempdima\box
\tw@}}}%
    {\ooalign{\box\tw@ \cr \box\z@}}}
\makeatother

\DeclarePairedDelimiter\floor{\lfloor}{\rfloor}

\newcommand{\eqdef}{=\mathrel{\mathop:}}
\newcommand\Set[2]{\{#1~\mid~#2\}}

\usepackage{natbib}
\usepackage{multirow}

\usepackage[colorlinks=true,
            linkcolor=blue,
            urlcolor=blue,
            citecolor=blue]{hyperref}

\usepackage[utf8]{inputenc}
\definecolor{fondpaille}{cmyk}{0,0,0.1,0}

\newcommand{\holdercon}{W_1}
\newcommand{\tvcon}{W_2}

\newcommand{\denslower}{p_{\text{min}}}
\newcommand{\tvlb}{c_{\text{TV}}}

\newcommand*{\TV}{\operatorname{TV}}
\newcommand*{\var}{\operatorname{var}}

\begin{document}

\begin{center}
{\LARGE Minimax Optimal Conditional Density Estimation under \\ 
\vspace{.3cm}
Total Variation Smoothness}

{\large
\begin{center}
\begin{tabular}{ccc}
Michael Li & Matey Neykov & Sivaraman Balakrishnan
\end{tabular}
\end{center}}

{Department of Statistics \& Data Science\\ Carnegie Mellon University\\Pittsburgh, PA 15213\\[2ex]\texttt{\{mli4, mneykov, sbalakri\}@andrew.cmu.edu}}

\begin{abstract} 
This paper studies the minimax rate of nonparametric conditional density estimation under a weighted absolute value loss function in a multivariate setting. We first demonstrate that conditional density estimation is impossible if one only requires that $p_{X|Z}$ is smooth in $x$ for all values of $z$.
This motivates us to consider a sub-class of absolutely continuous distributions, restricting the conditional density $p_{X|Z}(x|z)$ to not only be H\"older smooth in $x$, but also be total variation smooth in $z$. We propose a corresponding kernel-based estimator and prove that it achieves the minimax rate. We give some simple examples of densities satisfying our assumptions which imply that our results are not vacuous. Finally, we propose an estimator which achieves the minimax optimal rate adaptively, i.e., without the need to know the smoothness parameter values in advance. Crucially, both of our estimators (the adaptive and non-adaptive ones) impose no assumptions on the marginal density $p_Z$, and are not obtained as a ratio between two kernel smoothing estimators which may sound like a go to approach in this problem.
\end{abstract}
\end{center}

\section{Introduction}\label{intro:section}



A significant yet challenging problem in statistical inference is how to learn from complex, multidimensional data. While the nonparametric regression problem of estimating conditional mean $\mathbb{E} (x|z)$ from an \textit{i.i.d.} sample of $ (X, Z)$ is well studied, the alternative problem of estimating the full conditional density $p_{X|Z}(x|z)$ remains largely unexplored. There is little literature studying minimax optimal conditional density estimation, and particularly not when both $X$ and $Z$ are in a multivariate setting. However, the advantages of estimating $p_{X|Z}(x|z)$ over just the conditional mean are numerous. Fundamentally, the conditional mean is a summary of the conditional density. It follows that conditional density yields more information about the data and can be more useful for subsequent analysis. This is especially important when there exists multi-modality, asymmetry, or heteroscedastic noise in $p_{X|Z}(x|z)$, in which case the conditional mean $\mathbb{E} (x|z)$ would be insufficient to explain the data and to do inference. Furthermore, the problem of nonparametric quantile regression \cite{takeuchi2006nonparametric} can be solved via conditional density estimation. Finally, when forecasting and making predictions in fields such as economics, conditional density has been proven to be a key component of interest \cite{filipovic2012conditional}. However, although the advantages of conditional density estimation are clear, it is a harder problem than conditional mean estimation, which in turn raises the need to impose stronger assumptions.

To the best of our knowledge in this paper we give the first matching minimax upper and lower bounds for conditional density estimation in a multivariate setting. Concretely, the problem we consider is the following. Suppose here and throughout the paper that we have random variables $X \in [0, 1]^{d_X}$, $Z \in [0, 1]^{d_Z}$, and $n$ independent and identically distributed (i.i.d.) observations $\mathcal{D}_n = \{(X_1,Z_1), \ldots, (X_n,Z_n)\}$, coming from a joint distribution $p_{X,Z}$ that is absolutely continuous with respect to the Lebesgue measure on $[0,1]^{d}$, where $d = d_X + d_Z$. Our goal is to estimate the conditional density $p_{X|Z}(x|z)$ with the estimate $\hat p_{X|Z}(x|z)$ where $x = (x_1, \ldots, x_{d_X}), z = (z_1, \ldots, z_{d_Z})$. We will focus on the following loss function
\begin{align}\label{loss:function:def}
\mathbb{E} \int \int |\hat p_{X|Z}(x |z) - p_{X|Z}(x |z)| p_{Z}(z) d x d z,
\end{align}
where $p_{Z}$ denotes the marginal density of $Z$, the expectation is taken with respect to $n$ i.i.d. samples from $p_{X,Z}$, and $d x$ and $dz$ are shorthands for $\prod_{i \in [d_X]} d x_i$, $\prod_{i \in [d_Z]} d z_i$ respectively. The loss function \eqref{loss:function:def} is largely inspired by the works \cite{bashtannyk2001bandwidth, izbicki2016nonparametric}, where the authors consider the squared version of this loss.
The $L_1$-based loss function that we use has several benefits, and has been argued for in past works (see, for instance,~\cite{devroye2012combinatorial,devroye1985nonparametric} in the context of density estimation, and~\cite{balakrishnan2019hypothesis,balakrishnan2018hypothesis} in the context of density testing).
The $L_1$-based distance metric induced by this loss function is invariant to monotonic transformations, and in contrast to the $L_2$-based distance,
closeness in this distance has a clear probabilistic interpretation. Equivalently, the loss function we define may be interpreted as the $L_1$ distance between the joint distribution $\hat p_{X|Z}(x |z)p_{Z}(z)$ and the joint distribution $ p_{X, Z}(x, z)$. Furthermore, we can decompose the loss function into two parts. First we have $ \int |\hat p_{X|Z}(x |z) - p_{X|Z}(x |z)| d x$ which is equal to the $L_1$ distance between the estimated density and the target density. Next we weigh this distance by $p_{Z}$ to stress on the regions where $Z$ is more common, and downweight regions where $Z$ is less common. An important point that is worth making is that in this work we do not impose any assumptions on $p_Z$, which is enabled by the fact that the true density $p_Z$ is present in the loss function. Hence our estimators can handle situations where $p_Z$ may be non-differentiable and not even continuous. This is in stark contrast with an approach that one may be willing to take, i.e., to assume that $p_Z$ is H\"{o}lder smooth and estimate the conditional distribution as a ratio between kernel smoothed estimators of the joint $p_{X,Z}$ and the marginal $p_Z$.

We note that minimax rates with respect to this loss function have not been previously studied in the literature, and in fact any analysis of the minimax rates of conditional density estimation is scarce. The closest minimax analysis is given by \citet{efromovich2007conditional}, where the author studied minimax rates of conditional density estimation under an unweighted squared loss function. Unlike in the present work, \cite{efromovich2007conditional} only focused on the one dimensional setting, i.e., when $d_X = d_Z = 1$. Additionally, there exist more significant differences in the assumptions made and the problem settings, which will be elaborated on later.

\subsection{Relevant Literature}

In this section we review some of the relevant literature. In a classical work, \citet{rosenblatt1969conditional} proposed a kernel based estimate of $p_{X,Z}$ and $p_{Z}$ and combined them using the formula $p_{X|Z} = \frac{p_{X,Z}}{p_{Z}}$. Assuming that $p_{X | Z}$, $p_{Z}$ and the conditional mean of $X | Z$ have continuous second derivatives, \citet{hyndman1996estimating} analyzed the mean integrated squared error of a ratio between two kernel smoother estimators in the $d_X = 1$ dimensional case. \citet{bashtannyk2001bandwidth} looked into optimal bandwidth selection in the aforementioned kernel smoother estimate. \citet{fan1996estimation} used locally polynomial regression to develop nonparametric estimate of the conditional density function in nonlinear dynamical systems. In a follow-up work, \citet{fan2004crossvalidation} used cross-validation to select the bandwidth of the double-kernel estimator developed by \citet{fan1996estimation}. \citet{hall2004cross} used cross-validation to automatically reduce the number of relevant covariates when estimating the conditional density, but they did not study the minimax rates of estimation. In a related work, \citet{hall2005approximating} proposed a different method for estimating the density using dimension reduction. \citet{hall1999methods} studied methods for conditional distribution estimation based on parametric and nonparametric techniques, including a logistic model and a Nadaraya-Watson estimator. A different method using dimension reduction was proposed by \citet{efromovich2010dimension} where the author used an orthogonal series based approach. Chagny \cite{chagny2013warped} used an expansion of a ``warped'' conditional density onto a space spanned by orthonormal bases. Recently, \citet{cevid2020distributional} studied conditional density estimation using an adapted Random Forest algorithm. In conclusion, although there has been some work on conditional density estimation, the minimax optimal rate is an open question. In this paper we address this question for the loss function \eqref{loss:function:def} under certain smoothness assumptions on the conditional distributions $p_{X | Z = z}$.

\subsection{Summary of Results}


We begin by showing that conditional density estimation is impossible if one does not impose sufficient assumptions on the class of distributions. In particular, assuming that $p_{X|Z}$ is smooth in $x$ for all $z$ is not enough and further assumptions are needed. We formally prove this fact by arguing that for any sample size $n \in \mathbb{N}$, there exists a finite class of distributions whose conditional densities $p_{X|Z}$ are H\"older smooth (see Definition \ref{holder:smooth:def}) in $x$ for all $z$, for which the worst case loss is bounded from below by a constant. This result motivates the assumptions that we impose next.

We formalize a class of distributions $\mathcal{P}_{\beta, \gamma}$ consisting of conditional densities that are H\"older smooth with smoothness $\beta$ in $x$, and $\gamma$-total variation ($\gamma$-TV) smooth in $z$ (see Definition \ref{TV:smooth:def}). We show the following result:
\begin{align*}
\inf_{\hat p}\sup_{p \in \mathcal{P}_{\beta, \gamma}} \mathbb{E}_p \int \int |\hat p_{X|Z}(x |z) - p_{X|Z}(x |z)| p_{Z}(z) d x d z \asymp n^{\frac{-1}{\beta^{-1} d_X + \gamma^{-1} d_Z + 2}},
\end{align*}
where $\asymp$ means equality up to constant factors, and $\mathbb{E}_p$ is the expectation over $n$ i.i.d. samples, each of which comes from the distribution $p$. This minimax rate is achieved by a kernel-based estimator, which is defined in \eqref{holder:estimate}. Furthermore, observe that there is a curse of dimensionality, where the dimensions $d_X, d_Z$ may have different effects on the rate depending on the corresponding smoothing parameters $\beta$ and $\gamma$. 

Finally, we devise an adaptive estimator to achieve the minimax optimal rate without the need to know the values of the smoothness parameters $\beta$ and $\gamma$ in advance. Our estimator is based on the work of Yatracos \cite{yatracos1985rates}, but requires delicate care and crucial modifications, since we do not possess knowledge of, and are not willing to make any assumptions on the marginal $p_Z$.



\subsection{Notation}
The following notations will be used throughout the paper. We use $p_{X,Z} = p_{X | Z} \cdot p_{Z}$ to denote any joint distribution (and density function) of the pair of random variables $(X,Z)$. We also use $p_{X | Z}(x | z)$, $p_{X | Z = z}$ to denote the conditional density function and the conditional distribution of $X | Z = z$ respectively, and $p_{Z}$ to denote the marginal distribution (and density function) of $Z$. For an integer $n$ we use the shorthand $[n] = \{1,\ldots, n\}$. 

We also use multi-index notations. Suppose we have vectors $x\allowdisplaybreaks = (x_1, \ldots, x_{d_X})$, $\alpha = (\alpha_1, \ldots, \alpha_{d_X})$ such that $x \in \mathbb{R}^{d_X}, \alpha \in \mathbb{R}_+^{d_X}$, where $\RR_+ = \{x \in \RR | x \geq 0\}$, then we have
\begin{align*}
\|\alpha\|_1 = \sum_{i=1}^{d_X} |\alpha_i|, \quad \alpha! = \prod_{i=1}^{d_X} \alpha_i !, \quad x^{\alpha} = \prod_{i=1}^{d_X} x_i^{\alpha_i}.
\end{align*}
Furthermore
\begin{align*}
D^{\alpha}f = \frac{\partial^{\|\alpha\|_1} f}{\partial x_1^{\alpha_1} \dots \partial x_{d_X}^{\alpha_{d_X}}}.
\end{align*}
We let ${\lfloor}\beta{\rfloor}$ denote the greatest integer strictly less than the real number $\beta$. We also use $\lesssim, \gtrsim$ to mean inequalities up to universal constants, and we write $f(n) \asymp g(n)$ if both $f(n) \lesssim g(n)$ and $f(n) \gtrsim g(n)$ hold. 

\section{Impossibility}
In this section, we will show that it is, in general, impossible to estimate the conditional density at a reasonable rate unless some assumptions on the class of distributions are imposed. 
Importantly, we show that even if one is willing to assume that $p_{X|Z}$ is smooth in $x$ for all $z$, it is still insufficient, and more assumptions are needed. 
Intuitively, when the conditioning variable $Z$ has a continuous density we observe no replicates (multiple samples with identical $Z$ values) and it is necessary to impose that the conditional densities $p_{X|Z}$ 
are smooth in $z$ in order to reliably estimate $p_{X|Z}$. Our impossibility result, Theorem~\ref{impossible:thm}, formalizes this intuition.

As detailed in the introduction, for an estimate $\hat p_{X|Z}(x |z)$, based on a dataset $\cD_n = \{(X_1, Z_1), \allowbreak \ldots, (X_n, Z_n)\}$ with $n$ observations and a density $p_{X|Z}(x |z)$, we will use the loss function \eqref{loss:function:def}. In order to formally state our result, we first define H\"older smoothness. 
\begin{definition}[H\"older smoothness]\label{holder:smooth:def}
We say that the collection of conditional densities $p_{X|Z}(x|z)$ for $Z \in [0,1]^{d_Z}$ is H\"older smooth with some constant $\holdercon$ and smoothness $\beta$, where $\beta, \holdercon$ are positive numbers, if it is $\ell = {\lfloor}\beta{\rfloor}$ times differentiable, and for all $x, x' \in [0,1]^{d_X}$, $z \in [0,1]^{d_Z}$ satisfies
\begin{align*}
\sup_{\alpha} | D^{\alpha} p_{X|Z}(x|z) - D^{\alpha} p_{X|Z}(x'|z)| \leq \holdercon \|x - x'\|_1^{\beta - \ell}, \quad \text{for all } \: \alpha \: \text{such that} \: \|\alpha\|_1 = \ell, \alpha \in \mathbb{N}_{0}^{d_X},
\end{align*}
where $\alpha = (\alpha_1,...,\alpha_{d_X})$ and  $\mathbb{N}_{0} = \mathbb{N} \cup \{0\}$. \end{definition}

We then have the following result:
\begin{theorem}[Impossibility of Conditional Density Estimation]\label{impossible:thm} Let the sample size $n$ be any fixed integer. 
Then for any constants $\beta$, $\holdercon$ and $\varepsilon > 0$, there exists a finite class of distributions $\cC(\beta, \holdercon)$ (whose cardinality depends on $n$ and $\varepsilon$) on $[0,1]^{d_X + d_Z}$ satisfying the following three properties:
\begin{itemize}
\item[i.] the marginal $Z$ densities are absolutely continuous with respect to the Lebesgue measure on $[0,1]^{d_Z}$, with density equal to $p_{Z}$ (which can be specified by the user), 
\item[ii.] the conditional distributions $p_{X|Z}$ are H\"older smooth with constant $\holdercon$ and smoothness $\beta$, and
\item[iii.] the following inequality holds
\begin{align*}
\inf_{\hat p} \max_{p \in \cC(\beta, \holdercon)} \EE_{p} \int \int |\hat p_{X|Z}(x |z) - p_{X|Z}(x |z)| p_{Z}(z) d x d z \geq \kappa - \epsilon, 
\end{align*}
where $\kappa$ is some positive constant that depends on $\beta$ and $\holdercon$. More specifically, if $\beta \leq 1$ and $\holdercon < \frac{2}{d_X^{\beta}}$, we have $\kappa = \frac{\holdercon^2 d_X^{(2\beta - 1)}}{144}$, and otherwise $\kappa = \frac{1}{36 d_X}$.
\end{itemize}
\end{theorem}

{\noindent \bf Remarks: } 
\begin{enumerate}
\item It is important to note that this result holds for any arbitrary marginal $Z$ density $p_Z$ (i.e. it can be chosen to be arbitrarily smooth and known to the statistician). Our result shows that consistent conditional density estimation is impossible when
the only assumptions made are that the marginal density $p_Z$, and the conditional densities $p_{X|Z}$ are smooth (no matter how smooth they are). 
\item A straightforward extension of our proof shows that one can further relax the condition on the marginal $p_Z$. If the absolutely continuous part of $Z$ has probability mass at least $\theta$ for some $\theta > 0$ then an identical argument will show that the minimax error can be made arbitrarily close to $\kappa \theta.$ 
\end{enumerate}
\noindent Here we provide a sketch of the proof, while the full proof is deferred to Section~\ref{sec:impossibleproof}. 

\begin{proof}[Proof Sketch]
We define a ``null'' distribution by taking $p_{X|Z}$ to be uniform on $[0,1]^{d_X}$ for all values of $Z$. 
We then construct a family of ``alternate'' distributions which are perturbations of the null distribution constructed in the following way.
We first construct a pair of smooth distributions $p_1, p_2$ such that they yield that uniform distribution when mixed with equal weights, but which are individually sufficiently far from uniform. We divide the support of $Z$ into many small intervals, and in each interval, we randomly (with equal probability) set $p_{X|Z}$ to be either $p_1$ or $p_2$. This constructs a large family of possible alternate distributions.
We then argue that with high probability each sample point $Z_i$ falls into different intervals, and show that this in turn makes it impossible to distinguish whether the samples came from the null distribution or the uniform mixture over the possible alternate distributions.
\end{proof}

Theorem \ref{impossible:thm} illustrates that if we only assume that $p_{X|Z}$ is smooth in $x$ for all $z$ (e.g. H\"older smooth), we can construct a finite collection of distributions such that any estimate will produce an expected error of at least $\kappa$ in the worst case sense. Importantly, the proof makes use of the fact that we do not see replications of the densities $p_{X|Z}$ for different $Z$ values. We can remedy this by assuming that the distributions $p_{X|Z = z}$ vary smoothly with $z$. In the following section, we do so by imposing a Total Variation smoothness assumption on $z$, and show that under such conditions we can obtain reasonable (and minimax-optimal) bounds (i.e., bounds decreasing with the sample size) on the loss function. Intuitively, this happens since with additional smoothness assumptions, one can group observations whose $Z_i$ values are close, while this strategy is unavailable in the general setting.

\section{Upper Bound}\label{upper:bound:section}
In this section we propose an estimate $\hat{p}_{X|Z}(x|z)$ under certain smoothness assumptions on the class of distributions. Formally, suppose we have a joint distribution of two variables $(X,Z)$: $p_{X,Z}(x,z)$ where $x = (x_1, \ldots, x_{d_X}), z = (z_1, \ldots, z_{d_Z})$ and $X \in [0,1]^{d_X}$, $Z \in [0,1]^{d_Z}$. We assume that the conditional density $p_{X|Z}(x|z)$ satisfies H\"older smoothness in $x$ (see Definition \ref{holder:smooth:def}) and the following $\gamma$-total variation ($\gamma$-TV) smoothness in $z$. We denote the class of densities which satisfy our smoothness conditions by $\mathcal{P}_{\beta, \gamma}$.
\begin{definition}[$\gamma$-TV smoothness]\label{TV:smooth:def}
We say that the distribution is $\gamma$-total variation ($\gamma$-TV) smooth if the following inequality holds for some $0 < \gamma \leq 1$, and for all $x \in [0,1]^{d_X}$ and $z, z' \in [0,1]^{d_Z}$:
\begin{align*}
\|p_{X|Z = z} - p_{X|Z = z'} \|_1 \leq \tvcon\|z - z'\|_1^{\gamma},
\end{align*}
for some sufficiently large constant $\tvcon$.
\end{definition}
In the above, the $L_1$ distance between probability densities (equal to $2$ times the TV distance) is defined as: 
\begin{align*}
\|p_{X|Z = z} - p_{X|Z = z'} \|_1 = 2 \TV(p_{X|Z = z}, p_{X|Z = z'}) = \int |p_{X|Z }(x | z) - p_{X|Z }(x | z')| dx. 
\end{align*}
In other words, TV smoothness requires that the distributions $p_{X|Z = z}$ vary smoothly with $z$ in the $L_1$ sense. This assumption is inspired by \cite{neykov2020minimax}, where the authors used a similar assumption to establish the minimax rate for conditional independence testing. Furthermore, $\gamma \leq 1$ is required due to the following lemma:
\begin{lemma}\label{gamma:smaller:than:one}
Suppose $\gamma > 1$, and the inequality $\|p_{X|Z = z} - p_{X|Z = z'} \|_1 \leq W_2\|z - z'\|_1^{\gamma}$ from Definition \ref{TV:smooth:def} holds. Then it must be that $p_{X|Z = z} \equiv p_{X|Z = z'}$ for all $z, z' \in [0,1]^{d_Z}$.
\end{lemma}
\begin{proof}
Fix any two points $z, z'$ in $[0, 1]^{d_Z}$. Take $\alpha_j = \frac{j}{k + 1}$, $j = 0,1,\ldots, k + 1$. Let $z_j = \alpha_j z + (1 - \alpha_j) z'$. Then by $\gamma$-TV smoothness with $\gamma > 1$ we have
\begin{align*}
\TV(p_{X|Z = z}, p_{X|Z = z'}) \leq \sum_{i=0}^k \TV(p_{X|Z = z_i}, p_{X|Z = z_{i + 1}}) \leq W_2 \left(\frac{\|z-z'\|_1}{k+1}\right)^{\gamma} (k + 1).
\end{align*}
Taking $k \to \infty$ lets us conclude that $p_{X|Z = z} = p_{X|Z = z'}$ as desired.
\end{proof}

Finally, the estimator we propose under the H\"{o}lder smoothness assumption makes use of kernels. Below we define a class of kernels that can be used in the estimator to achieve the minimax optimal rate.
\begin{definition}[Appropriate Kernels]\label{kernel:def}
We say that a kernel $K : \mathbb{R}^{d_X} \to \mathbb{R}$ is appropriate if
\begin{align*}
\int K(u) du = 1, \quad \int u^{\alpha} K(u) du = 0, \: \text{for all } \: \bm{\alpha } \: \text{such that} \: \|\alpha\|_1 \leq \ell, \alpha \in \mathbb{N}_{0}^{d_X},
\end{align*}
where $\mathbb{N}_{0} = \mathbb{N} \cup \{0\}$. In addition the kernel should satisfy
\begin{align*}
\int K^2(u)du < \infty
\end{align*}
and
\begin{align*}
\int |K(u)| \cdot |u^{\alpha}|< \infty,  \: \text{for all } \: {\alpha } \: \text{such that} \: \|\alpha\|_1 \leq \beta, \alpha \in \RR_{+}^{d_X},
\end{align*}
where $\RR_+ = \{x \in \RR ~|~ x \geq 0\}$.
\end{definition}

%
%
%
%
%

Importantly, appropriate kernels do exist, and one method of constructing them is detailed in Lemma \ref{kernel:construct} below.

\begin{lemma}[Appropriate Kernels' Construction]\label{kernel:construct} 
We can construct appropriate kernels $K : \mathbb{R}^{d_X} \to \mathbb{R}$ using a product kernel $K(u) = \prod K_i(u_i)$, where each $K_i$ is a kernel of order $\ell$ as defined in \cite[Proposition 1.3]{tsybakov09introduction}.
\end{lemma}

We prove this lemma in Appendix \ref{app:upper}. We now formally define the estimator. Recall that we are interested in estimating the conditional density $p_{X|Z}(x |z)$ with the estimate $\hat p_{X|Z}(x |z)$ under the loss function \eqref{loss:function:def}. We propose a histogram-type estimator that uses kernel smoothing. 
Namely, bin $[0,1]$ into intervals $A_1,\ldots, A_m$ of equal length ($m^{-1}$), and consider the hyper-rectangles created from the Cartesian product of such intervals over $[0,1]^{d_Z}$.
We define the following shorthand notations: let $\bar{j} = (j_1, ..., j_{d_Z}) \in [m]^{d_Z}$ denote the bin indices for some $d_Z$ dimensional hyper-rectangle. Then $A_{\bar{j}} = \prod_{k=1}^{d_Z} A_{j_k}$ denotes that hyper-rectangle itself, where $\prod$ stands for the Cartesian product between sets. Finally, define the estimate

\begin{align}\label{holder:estimate} 
\hat{p}_{X|Z}(x|z) = \sum_{\bar{j} \in [m]^{d_Z}} \mathbbm{1}\left(z \in A_{\bar{j}} \right) \frac{\sum_{i \in [n]} \mathbbm{1}(Z_i \in A_{\bar{j}}) K(\frac{X_i - x}{h})}{h^{d_X} \sum_{i \in [n]} \mathbbm{1} (Z_i \in A_{\bar{j}})},
\end{align}
where $K: \mathbb{R}^{d_X} \to \mathbb{R}$ is an appropriate multi-dimensional kernel as in Definition \ref{kernel:def} and $\frac{0}{0}$ is understood as $0$. Note that we only apply binning to $Z$ here, and apply kernel smoothing to $X$. Our estimator \eqref{holder:estimate} need not be a proper density since it need not be positive, but we provide a simple modification below. This modified estimator has the same properties as $\hat p_{X|Z}$ but is in fact a proper density. This will play an important role when we devise our adaptive estimator in Section \ref{yatracos:section}. 

\begin{theorem}[H\"older Upper Bound] \label{upper:bound:holder:thm}
Suppose that $p_{X,Z} \in \mathcal{P}_{\beta, \gamma}$. Using the estimate \eqref{holder:estimate} with an appropriate kernel as per Definition \ref{kernel:def}, and selecting the parameters  $h \asymp n^{\frac{-1}{d_X + d_Z \beta \gamma^{-1} + 2\beta}}$ and $m \asymp n^{\frac{\beta}{d_X + d_Z \beta \gamma^{-1} + 2\beta}}$ (for some appropriately selected constants) we have that 
\begin{align}\label{eq:thm2}
\mathbb{E} \int \int |\hat p_{X|Z}(x |z) - p_{X|Z}(x |z)| p_{Z}(z) d x d z \lesssim n^{\frac{-1}{\beta^{-1} d_X + \gamma^{-1} d_Z + 2}} =: r_n(\beta, \gamma, d_X, d_Y).
\end{align}
\end{theorem}
\noindent Theorem \ref{upper:bound:holder:thm} provides an upper bound for the estimator \eqref{holder:estimate}, and in Section \ref{lower:bound:section} we will derive a matching lower bound. Combining the two proves that \eqref{eq:thm2} is in fact the minimax optimal rate, and therefore no estimator can do better than $\hat p_{X | Z}$ up to constant factors. We defer the proof of Theorem \ref{upper:bound:holder:thm} to Section \ref{proof:of:upper:bound}. Roughly, the proof consists of a ``bias'' and ``variance'' decomposition (in quotation marks due to the fact that the two terms in the decomposition are not exactly bias and variance since our loss function is not the squared loss)  and carefully controlling both ensures that the rate exhibited in \eqref{eq:thm2} holds. \\

{\noindent \bf Remarks: } 
\begin{enumerate}
\item Equation \eqref{eq:thm2} shows that the estimator we propose exhibits the curse of dimensionality.
In particular, the presence of $\beta$ and $\gamma$ smoothness parameters makes sense intuitively. Recall that the class of distributions $\mathcal{P}_{\beta, \gamma}$ assumes conditional densities that are $\beta$-H\"older smooth in $x$, and $\gamma$-TV smooth in $z$, which matches the effects observed here. We note that when holding both dimensions $d_X, d_Z$ constant, our estimator performs better for higher values of $\beta$ or $\gamma$ smoothness (i.e. smoother densities). Indeed, as $\beta$ increases, the effect of $d_X$ on the estimator's effectiveness diminishes, and when $\beta \to \infty$, the dimension of $X$ has no effect on the minimax rate at all. In addition, since our proof does not require the restriction $\gamma \leq 1$, we note that when we take $\gamma \to \infty$, the minimax rate simply reduces to that of classical unconditional density estimation of $x$ in a multivariate setting with H\"older smoothness assumptions. In fact, by Lemma \ref{gamma:smaller:than:one} any value of $\gamma > 1$ can be thought of as $\gamma \rightarrow \infty$. 
\item Once again, we would like to stress the fact that we make no assumptions on the marginal density of $Z$. This is in stark contrast to an approach one may be compelled to take, by first estimating the joint density $p_{X, Z}(x,z)$, then the marginal density of $z$ by integrating $x$ out, and dividing the two to arrive at an estimate of $p_{X|Z}$. This implies consistently estimating $p_{X, Z}(x,z)$ and $p_{Z}(z)$, which likely requires assumptions on both of these densities, whereas that is not needed in our case. In fact, Theorem \ref{example:greater:thm} provides examples of one such class of H\"older smooth densities where our approach is minimax optimal regardless of what $p_{Z}$ is, whereas the aforementioned approach will likely fail. 

\item Finally, as discussed earlier, there does not exist work closely comparable to ours. The most related paper is \cite{efromovich2007conditional}, where the author studied local minimax rates for conditional density estimation in a bivariate case (i.e. $d_X = d_Z = 1$). Furthermore, \cite{efromovich2007conditional} imposed vastly different assumptions (e.g. the conditional densities are assumed to belong to a perturbed Sobolev class), and utilized different loss functions. The difference in the two problem settings renders the comparison of the resulting minimax rates nonproductive.
\end{enumerate}



We now provide a modified estimator of \eqref{holder:estimate} to ensure that it is a proper density. Consider the following estimator: 
\begin{itemize}
\item if $\hat p_{X|Z = z} \not \equiv 0$ (which implies $\int \hat p_{X|Z}(x|z) dx = 1$), define $\bar p_{X|Z}(x |z) = C^{-1} (\hat p_{X|Z}(x |z))_+$ where $C = \int (\hat p_{X|Z}(x |z))_+ d x$, and for a function $f(x)$ we denote $(f(x))_+ = f(x)\mathbbm{1}(f(x) \geq 0)$; 
\item else if $\hat p_{X|Z = z} \equiv 0$, define $\bar p_{X|Z = z} \equiv 1$. 
\end{itemize} 

\begin{lemma}\label{lemma:ptilde:density}
$\bar p_{X|Z}(x |z)$ satisfies \eqref{eq:thm2}, 
and is a proper density.
\end{lemma}

The details of Lemma \ref{lemma:ptilde:density} are given in Appendix \ref{app:upper}. Finally, recall our result in equation \eqref{eq:thm2}, which shows an upper bound for the loss function \eqref{loss:function:def} by a quantity which is of the same order as $r_n := r_n(\beta, \gamma, d_X, d_Y)$. Now consider our modified estimator $\bar p_{X|Z}(x |z)$. We have the following result which is a simple consequence of Markov's inequality:

\begin{lemma}\label{markovs:ineq:lemma} For any $\epsilon > 0$ there exists a set $A_\epsilon$ satisfying $\PP(Z \in A_\epsilon) \geq 1 - \epsilon$, such that for all $z \in A_\epsilon$ we have
\begin{align*}
\TV(\bar p_{X|Z}(x |z), p_{X|Z}(x |z)) \lesssim \frac{1}{\epsilon}r_n.
\end{align*}
\end{lemma}

\begin{proof}
Since $\bar p_{X|Z}$ is a density this allows us to write 
\begin{align*}
\int | \bar p_{X|Z}(x |z) - p_{X|Z}(x |z) | dx = 2\TV(\bar p_{X|Z}(x |z), p_{X|Z}(x |z)).
\end{align*}
Let $f(z) :=\TV(\bar p_{X|Z}(x |z), p_{X|Z}(x |z))$, and note that by Theorem \ref{upper:bound:holder:thm} and Lemma \ref{lemma:ptilde:density} we know that 
\begin{align*}
\EE f(z) \lesssim r_n.
\end{align*}
By the Markov's inequality we have
\begin{align*}
\mathbb{P}\left(f(z) > \frac{1}{\epsilon} \mathbb{E} [f(z)]\right) \leq \epsilon.
\end{align*}
This implies that there exists a subset of the support of $Z$ -- call it $A_\epsilon$ -- which has probability of at least $1 - \epsilon$ to occur, and for all $z \in A_\epsilon$, $f(z) \leq \frac{1}{\epsilon} \mathbb{E}[f(z)] \lesssim \frac{1}{\epsilon} r_n$. This completes the proof.
\end{proof}
Lemma \ref{markovs:ineq:lemma} illustrates that we can estimate well an overwhelming majority (in terms of the $Z$ distribution) of conditional densities in terms of TV distance. Next we move on to establish a lower bound.

\section{Minimax Lower Bound}\label{lower:bound:section}

In this section we produce a minimax lower bound for the estimation problem with the loss function~\eqref{loss:function:def}. We recall the definition of the class $\mathcal{P}_{\beta, \gamma}$ in~Section~\ref{upper:bound:section}.

\begin{theorem}[H\"older Lower Bound] \label{lower:bound:holder:thm} For any $p_{Z}(z) \geq c$ where $c$ is some constant, we have that
\begin{align*}
\inf_{\hat p}\sup_{p \in \mathcal{P}_{\beta, \gamma}} \mathbb{E}_p \int \int |\hat p_{X|Z}(x |z) - p_{X|Z}(x |z)| p_{Z}(z) d x d z \gtrsim n^{\frac{-1}{\beta^{-1} d_X + \gamma^{-1} d_Z + 2}}.
\end{align*}
\end{theorem}
{\noindent \bf Remarks: } 
\begin{enumerate}
\item Theorem \ref{upper:bound:holder:thm} and this theorem together show that the proposed estimate \eqref{holder:estimate} achieves the minimax rate $n^{\frac{-1}{\beta^{-1} d_X + \gamma^{-1} d_Z + 2}}$ under the loss function \eqref{loss:function:def}. 
\item It is worth comparing and contrasting the results of this Theorem with our earlier impossibility result in Theorem~\ref{impossible:thm}. In rough terms, they convey the same basic intuition that when $\gamma$ is very small (or 0)
conditional density estimation is difficult (or impossible) in a minimax sense. This result is more quantitative, capturing in a more precise sense the dependence on $\gamma$. On the other hand, the result of Theorem~\ref{impossible:thm} is more
flexible, allowing essentially any marginal density $p_Z$ (not requiring it to be lower bounded by a constant), as long as the marginal density has some non-trivial absolutely continuous component (as discussed in the remarks following  Theorem~\ref{impossible:thm}).
\end{enumerate}
Here we provide a sketch of the proof for the minimax optimal lower bound. The full proof is deferred to Section~\ref{sec:lower}. 
\begin{proof}[Proof Sketch] 
We first define a class of conditional density functions
and show that their joint distributions indeed belong to $\mathcal{P}_{\beta, \gamma}$. Then, we apply Fano's inequality to derive the minimax lower bound.

The conditional density functions are defined as
\begin{align*}
p^{\Delta}_{X|Z}(x|z) & = 1 + \sum_{\bar{i} \in [m]^{d_X}} \sum_{\bar{j} \in [m]^{d_Z}} \Delta_{\bar{i}, \bar{j}} \prod_{k \in [d_X]} h_{i_k}(x_k) \prod_{k \in [d_Z]} g_{j_k}(z_k),
\end{align*}
where recall the shorthands $\bar{i} \in [r]^{d_X}, \bar{j} \in [m]^{d_Z}$ ($r, m$ are integers chosen later), and $\Delta_{\bar{i}, \bar{j}} \in \{ \pm 1 \}$. The intuition for such a construction is to add multiple small perturbations to the uniform conditional density function by using infinitely differentiable bump functions $h_{i_k}, g_{j_k}$. We proceed to verify that the constructed conditional density functions are indeed density functions (i.e. always positive and integrates to 1), and follow both the $\gamma$-TV smoothness condition as in Definition \ref{TV:smooth:def} and the H\"older smoothness condition as in Definition \ref{holder:smooth:def}.

In order to apply Fano's inequality \cite{Yu1997fano}, we first show that there exists a subset of the conditional density functions defined above, such that the distance between any pair as measured by the loss function is sufficiently large (more specifically, is lower bounded by some $\epsilon$). This is done by using Varshamov-Gilbert's construction \cite[Lemma 2.9]{tsybakov09introduction}. We then find an upper bound on the Kullback-Leibler (KL) divergence between any pair of our conditional density functions. Finally, applying Markov's inequality to the Fano's inequality allows us to express the minimax lower bound in terms of the distance lower bound and KL divergence upper bound we just derived. Making some optimal selection of parameter values completes the proof and produces the desired matching minimax optimal lower bound.
\end{proof}

Importantly, note that in the process of proving Theorem \ref{lower:bound:holder:thm}, we constructed a class of conditional density functions and showed that it satisfies all our assumptions. 
In order to better understand the class $\mathcal{P}_{\beta, \gamma}$, we develop further examples of densities belonging to this class in the next section.

\section{Examples} \label{examples:section}
In this section we provide examples of distributions which belong to $\mathcal{P}_{\beta, \gamma}$. Recall that this class of distributions requires conditional densities to be H\"older smooth (see Definition \ref{holder:smooth:def}) and $\gamma$-TV smooth (see Definition \ref{TV:smooth:def}). We already saw examples of such distributions in the proof of Theorem \ref{lower:bound:holder:thm}. Below we give two additional classes of examples for different values of the smoothness $\beta$.

\begin{theorem}[Examples for  $\beta > 1$]\label{example:greater:thm}
Suppose $g(x, z): [0,1]^d \mapsto \mathbb{R}$ is such that $g(x, z) \geq a > 0$ for some constant $a$, and is H\"older smooth with smoothness $\beta > 1$ in both $x$ and $z$, i.e.,
\begin{align*}
\sup_{\alpha}  | D^{\alpha} g(x,z) - D^{\alpha} g(x',z')| \leq C ( \| x - x'\|_1 + \| z - z'\|_1 )^{\beta - \ell},
\end{align*}
for all $\alpha$ such that $\|\alpha\|_1 = \ell, \alpha \in \mathbb{N}_{0}^{d_X}$, where $\ell = \floor{\beta}$. Then if $p_{X|Z}(x|z) = \frac{g(x,z)}{\int g(x,z) d x}$, we have $p_{X|Z} \in \mathcal{P}_{\beta, 1} \subseteq \mathcal{P}_{\beta, \gamma}$, for any $\gamma \leq 1$.
\end{theorem}

\begin{theorem}[Examples for $\beta \leq 1$]\label{example:less:thm}
Suppose $g(x,z): [0,1]^d \mapsto [-M,M]$ is a bounded function such that
\begin{align*}
|g(x,z) - g(x',z')| \leq C (\|x - x'\|_1^{\beta} + \|z - z'\|^\gamma_1). 
\end{align*}
Then if $p_{X|Z}(x|z) = \frac{\exp(g(x,z))}{\int \exp(g(x,z)) d x},$ we have $p_{X,Z} \in \mathcal{P}_{\beta, \gamma}$.
\end{theorem}

The proofs of Theorem \ref{example:greater:thm} and \ref{example:less:thm} are given in Appendix \ref{app:example}.

\section{Hyperparameter Tuning and Selection}\label{yatracos:section}
We have shown that our proposed estimate \eqref{holder:estimate} achieves the minimax-optimal rate $n^{\frac{-1}{\beta^{-1} d_X + \gamma^{-1} d_Z + 2}}$ under the loss function \eqref{loss:function:def}. 
However, in doing so we manually picked the values of the hyperparameters $h$ and $m$, which depend on the smoothness parameters $\beta$ and $\gamma$ of the true distribution. Here we introduce an adaptive method of selecting the hyperparameters without needing to assume the knowledge of $\beta$ and $\gamma$.

Towards the goal of hyperparameter tuning we first design and analyze a selection procedure for conditional density estimation which satisfies a type of oracle inequality with respect to the loss~\eqref{loss:function:def}. Given a collection of candidate conditional density estimates we devise a procedure which selects one which has nearly minimal loss. Our procedure is inspired by a minimum distance estimate described in the work of Yatracos~\cite{yatracos1985rates}, and further developed in the works~\cite{devroye2012combinatorial,devroye1985nonparametric}.
However, in contrast to these works our selection procedure only has access to conditional density estimates (as opposed to joint density estimates), and we aim to design a selection procedure tailored to the loss~\eqref{loss:function:def} (as opposed to the usual $L_1$ loss on the joint densities). Furthermore, our goal is to avoid smoothness assumptions which would be required to estimate the marginal of $Z$, and this 
necessitates careful modifications of the minimum distance procedure. 

We describe our oracle inequality in Theorem~\ref{thm:yatracos} and use this result to develop an adaptive conditional density estimate which achieves the same minimax-optimal rates as the estimate in~\eqref{holder:estimate} without knowledge of the smoothness parameters in Section~\ref{sec:tuning}.

\subsection{A Modified Selection Procedure for Conditional Density Estimates}
\label{sec:yatracos}

To begin with we consider the following setup. We are given access to a 
collection of conditional density estimates $\widehat{f}_1,\ldots,\widehat{f}_N$, which are either fixed, or estimated on a separate sample. Our goal is to select an estimate of 
(nearly) minimal loss.

We associate each estimate with an oracle joint distribution $\widetilde{f}_j(x,z) = \widehat{f}_j(x|z) p_Z(z)$ where $p$ is the unknown true density of the samples. 
Associated with each pair $(i,j)$ of density estimates we define the so-called Yatracos set:
\begin{align*}
A_{ij} = \{(x,z): \widetilde{f}_i(x,z) > \widetilde{f}_j(x,z)\}.
\end{align*}
We note that we can compute $A_{ij}$ even without access to the unknown density $p$. Denote the collection of such sets $\mathcal{A}$. 
For a set $A$ we let $A^z$ denote the subset with $Z = z$.

Given $n$ samples $\{(X_1,Z_1),\ldots,(X_n,Z_n)\}$  from $p$, 
we use the following minimum distance estimator:
\begin{align*}
\psi = \argmin_{\widehat{f}_j: j \in [N]} \sup_{A \in \mathcal{A}} \left| \frac{1}{n} \sum_{i=1}^n \int_{A^{Z_i}} \widehat{f}(x | Z_i) dx - \mathbb{P}_n(A) \right|.
\end{align*}
In rough terms our selection procedure compares, for each candidate $\widehat{f}_j$, an estimate of the mass of the Yatracos sets under $\widetilde{f}_j$ to an estimate of the population mass of these sets, selecting the candidate for which the largest discrepancy is smallest. 

We show the following result:
\begin{theorem}
\label{thm:yatracos}
With probability at least $1 - \delta$, 
\begin{align*}
\int_z \|\psi(x|z) - p(x|z)\|_1 p_Z(z) dz \leq 3 \min_{j} \int_z \|\widehat{f}_{j} - p(x|z)\|_1 p_Z(z) dz+  14\sqrt{ \frac{\log (N/\delta)}{n}}.
\end{align*}
\end{theorem}
{\noindent \bf Remarks:}
\begin{enumerate}
\item Our method and analysis are inspired by those of Yatracos \cite{yatracos1985rates}. The crucial insight of Yatracos is that when our goal is to select one of a collection of candidates, the supremum over the relatively small collection of Yatracos sets is adequate as a (statistically and computationally) tractable proxy for the supremum over all measurable sets (the TV/$L_1$ distance). 
\item The guarantee of the theorem is extraordinary in that we are able to obtain an oracle inequality with an excess error which scales as $1/\sqrt{n}$, despite the fact that accurately estimating the loss~\eqref{loss:function:def} of even a single estimate would require many more samples. Furthermore, the guarantee degrades only logarithmically in the number of estimates $N$ we are aiming to select from and in the failure probability $\delta$. This, for instance, will be important in the next section when we use the method to select from a large collection of candidate density estimates constructed using different values of the tuning parameter.
\item Although not the main focus of our paper, the computational costs of constructing the Yatracos sets and computing the minimum distance estimate are discussed extensively in~\cite{devroye2012combinatorial,devroye1985nonparametric}. At the expense of a slightly worse guarantee one might use a tournament-based selection rule which has a computational cost which scales linearly (as opposed to quadratically) in the number of estimates $N$.
\end{enumerate}

\subsection{Adaptive Conditional Density Estimation}
\label{sec:tuning}
With our previous result in place we now describe an adaptive conditional density estimate which achieves the rate $n^{\frac{-1}{\beta^{-1} d_X + \gamma^{-1} d_Z + 2}}$, without knowledge of the smoothness parameters
$\beta$ and $\gamma$. We assume throughout that $\beta$ and $\gamma$ are upper bounded by some (unknown, but fixed) universal constants. We split our sample in two halves, using one half to construct a collection of candidate density estimates, and the second half to select one of these candidates following the procedure in Section~\ref{sec:yatracos}.
In practice, one might choose instead to use cross-fitting, where we repeat this procedure swapping the roles of the two samples, and average the two resulting estimates. It is straightforward to show that our guarantees
continue to hold for the cross-fit variant as well.

We postulate two intervals where $h$ and $m$ are assumed to lie in respectively. 
Recall our choice of optimal hyperparameter values while proving Theorem \ref{upper:bound:holder:thm}: $h \asymp n^{\frac{-1}{d_X + d_Z \beta \gamma^{-1} + 2\beta}}$ and $m \asymp n^{\frac{\beta}{d_X + d_Z \beta \gamma^{-1} + 2\beta}}$. Based on this result we consider values of the tuning parameters in two sets $\mathcal{I}_1 = \{n^{-1/d_X}, 2 \times n^{-1/d_X}, \ldots, 2^{\lceil \log_2 n^{1/d_X}\rceil} n^{-1/d_X}\}$, and $\mathcal{I}_2 = \{1, 2, 4, \ldots, 2^{\lceil (\log_2 n)/2 \rceil} \}$. We consider all pairs of tuning parameters $(h, m) \in \mathcal{I}_1 \times \mathcal{I}_2$, noting that there are at most $N := \mathcal{O}(\log^{2} n)$ such choices. 
For 
each possible hyperparameter combination, we compute our conditional density estimate $\bar p_{j}  (x | z)$, for $j \in [N]$ (we use the truncated and renormalized estimate analyzed in Lemma~\ref{lemma:ptilde:density}). 
At least one of these estimates achieves the minimax rate of $n^{\frac{-1}{\beta^{-1} d_X + \gamma^{-1} d_Z + 2}}$ and it thus only remains to select a sufficiently good candidate.

We apply the Yatracos procedure from Section~\ref{sec:yatracos} to select a candidate $\widehat{\psi}$ using the second half of the sample, and we obtain the following result. 
\begin{theorem}[Adaptive Conditional Density Estimation] \label{thm:adaptive}
Suppose that $p_{X,Z} \in \mathcal{P}_{\beta, \gamma}$. The tuning-parameter free procedure described above yields an estimate $\widehat{\psi}$ such that,
\begin{align*}
\mathbb{E} \int \int |\hat \psi_{X|Z}(x |z) - p_{X|Z}(x |z)| p_{Z}(z) d x d z \lesssim n^{\frac{-1}{\beta^{-1} d_X + \gamma^{-1} d_Z + 2}}.
\end{align*}
\end{theorem}
{\noindent \bf Remarks:} 
\begin{enumerate}
\item The proof of this result is fairly straightforward given the result of the previous section, and only requires us to convert the high-probability bound from Theorem~\ref{thm:yatracos} so that we may use our
previously derived upper bound in Theorem~\ref{upper:bound:holder:thm}. We present the details in Section~\ref{sec:adaptive_proof}.
\item We note that, in contrast to works on pointwise adaptation over smoothness classes \cite{lepski1997optimal} where typically a logarithmic price is unavoidable, 
we design an adaptive estimator which achieves the (oracle) minimax-rate (of Theorem~\ref{lower:bound:holder:thm}).
\end{enumerate}

\section{Proofs}
In this section we present the proofs of the main results of our paper, deferring remaining technical aspects to the Appendix. 
\subsection{Proof of Theorem~\ref{impossible:thm}}
\label{sec:impossibleproof}
We start by constructing a joint distribution $p_{X,Z}(x, z)$ such that $p_{X|Z}$ is uniform on $[0,1]^{d_X}$ for all values of $Z$ (i.e. $X$ is independent of $Z$), and the marginal of $Z$ is equal to $p_{Z}$ which is specified by the user. We will now construct multiple distributions out of $p_{X,Z}(x, z)$. Take disjoint Borel measurable sets $C_1, \ldots, C_m \subset [0,1]^{d_Z}$ such that $\int_{C_i} p_{Z}(z) dz = m^{-1}$ and $\cup_{i \in [m]}C_i = [0,1]^{d_Z}$.

For a given $\beta$ and $\holdercon$, we construct two distributions $p_1, p_2$ which are H\"older smooth with constant $\holdercon$ and smoothness $\beta$, and have sufficiently large total variation (TV) distance from the uniform distribution. In addition, their mixture distribution with equal weights produces $p_{X|Z}$, i.e. 
\begin{align*}
\frac{1}{2} p_1 + \frac{1}{2} p_2 = p_{X|Z} = U([0,1]^{d_X}).
\end{align*}
For a constant $0 < c < 1$ which we will set appropriately in the sequel, we define $p_1, p_2$ as linear functions of the $x_i$ as follows:
\begin{align*}
p_1(x) = 2(1 - c) \sum_{i \in [d_x]} \frac{x_i}{d_X} + c, \quad p_2(x) = 2 - p_1(x).
\end{align*}
The following result develops some properties of the distributions $p_1$ and $p_2$.
\begin{lemma} \label{lemma:p1p2}
\begin{enumerate}
\item The densities $p_1, p_2$ are positive, integrate to 1, and satisfy the property that $\frac{1}{2} p_1 + \frac{1}{2} p_2 = U([0,1]^{d_X})$. 
\item When $\beta \leq 1$ the densities are H\"{o}lder with $\holdercon = 2(1 - c)/d_X^\beta$ and if $\beta > 1$ they are H\"{o}lder for any value $\holdercon \geq 0$.
\item Furthermore, $\TV(U([0,1]^{d_X}), p_1) = \TV(U([0,1]^{d_X}), p_2) \geq \frac{(1 - c)^2}{18d_X} \eqdef \tvlb.$
\end{enumerate}
\end{lemma}

Now, for a $\Delta = (\Delta_1, \ldots, \Delta_m) \in \{1,2\}^m$, construct the distribution $p^{\Delta}_{X,Z}(x,z)$ which has the same marginal distribution on $Z$ as $p^{}_{X,Z}(x,z)$, i.e., $p_{Z}$, and the conditional distributions defined as follows:
for $z \in C_j$ we have
\begin{align*}
p^\Delta_{X|Z}(\cdot | z) = p_{\Delta_j}( \cdot).
\end{align*}
We know from Lemma \ref{lemma:p1p2} that $\TV(U([0,1]^{d_X}), p_1) = \TV(U([0,1]^{d_X}), p_2) \geq \tvlb$ where $\tvlb$ is some positive constant, so it follows that
\begin{align}\label{distance:between:null:and:alt}
\int \int | p_{X|Z}(x|z) -  p^{\Delta}_{X|Z}(x|z) | p_{Z}(z) dx dz  = \int 2\TV(U([0,1]^{d_X}), p_1) p_{Z}(z) dz \geq 2\tvlb > 0.
\end{align}

Next the proof will emulate a classical reduction scheme from estimation to a testing problem. This reduction is similar to the one described in Section 2.2 of \cite{tsybakov09introduction}, yet there are differences hence we provide full details. For brevity note that our loss function is
\begin{align*}
\|\hat  p_{X|Z} p_{Z} - q_{X|Z} p_{Z}\|_1 = \int \int |\hat p_{X|Z}(x |z) - q_{X|Z}(x |z)| p_{Z}(z) d x d z. 
\end{align*}
By Markov's inequality, we have that:
\begin{align}\label{simple:ineq:markov}
\EE \|\hat  p_{X|Z} p_{Z} - q_{X|Z} p_{Z}\|_1 \geq \tvlb \PP\bigg(\|\hat  p_{X|Z} p_{Z} - q_{X|Z} p_{Z}\|_1 \geq \tvlb \bigg). 
\end{align}
Define the finite class of distributions $\cC(\beta, \holdercon) = \{p_{X|Z}, \{p^{\Delta}_{X|Z}\}_{\Delta \in \{1, 2\}^m}\}$, and for ease of notation enumerate the elements of $\cC(\beta, \holdercon)$ by $p_0 = p_{X|Z}$, $p_{\Delta + 1} = p^{\Delta}_{X|Z}$ where with a slight abuse of notation we refer to $\Delta$ as the integer with binary representation $\Delta$. Thus the cardinality of the set $|\cC(\beta, \holdercon)| = 2^m + 1$. In light of \eqref{simple:ineq:markov}, it follows that in order to lower bound the quantity
\begin{align*}
\inf_{\hat p_{X|Z}} \max_{i \in \{0\}\cup[2^m]} \EE_{p_{i}p_{Z}} \|\hat  p_{X|Z} p_{Z} - p_i p_{Z}\|_1,
\end{align*}
it suffices to control 
\begin{align*}
\inf_{\hat p_{X|Z}} \max_{i \in \{0\}\cup[2^m]} \PP_{p_i p_{Z}}\bigg(\|\hat  p_{X|Z} p_{Z} - p_i p_{Z}\|_1 \geq \tvlb \bigg),
\end{align*}
where we are indexing the expectation and the probability to stress that the distribution of the data is generated under the distribution $p_i p_{Z}$ (importantly note that we have $n$ i.i.d. observations from $p_i p_{Z}$). Next we notice that
\begin{align*}
\PP_{p_{0} p_{Z}}\bigg(\|\hat  p_{X|Z} p_{Z} - p_0 p_{Z}\|_1 \geq \tvlb \bigg) \geq \PP_{p_0 p_{Z}} (\psi^* \neq 0),
\end{align*}
and 
\begin{align*}
\PP_{p_{i} p_{Z}}\bigg(\|\hat  p_{X|Z} p_{Z} - p_i p_{Z}\|_1\geq \tvlb \bigg) \geq \PP_{p_i p_{Z}} (\psi^* = 0),
\end{align*}
where 
\begin{align*}
\psi^* = \argmin_{0 \leq i \leq 2^m} \|\hat  p_{X|Z} p_{Z} - p_ip_{Z}\|_1,
\end{align*}
and the above two inequalities follow by the triangle inequality and \eqref{distance:between:null:and:alt}. We conclude that
\begin{align*}
\MoveEqLeft \inf_{\hat p_{X|Z}} \max_{i \in \{0\}\cup[2^m]} \PP_{p_i p_{Z}}\bigg(\|\hat  p_{X|Z} p_{Z} - p_i p_{Z}\|_1 \geq \tvlb \bigg) \\
& \geq \inf_{\psi} \max(\PP_{p_0 p_{Z}} (\psi \neq 0), \max_{i \in [2^m]} \PP_{p_i p_{Z}} (\psi = 0)),
\end{align*}
where the $\inf$ is taken over all measurable test functions with values in the set $\{0\}\cup [2^m]$. Using the fact that the $\max$ is bigger than the average we have
\begin{align}
\MoveEqLeft \inf_{\hat p_{X|Z}} \max_{i \in \{0\}\cup[2^m]} \PP_{p_i p_{Z}}\bigg(\|\hat  p_{X|Z} p_{Z} - p_i p_{Z}\|_1 \geq \tvlb \bigg) \nonumber\\
& \geq \inf_{\psi} \max(\PP_{p_0 p_{Z}} (\psi \neq 0), \max_{i \in [2^m]} \PP_{p_i p_{Z}} (\psi = 0)) \nonumber \\
& \geq  \inf_{\psi} \max\bigg(\PP_{p_0 p_{Z}} (\psi \neq 0), \frac{1}{2^m} \sum_{i \in [2^m]} \PP_{p_i p_{Z}} (\psi = 0)\bigg).\label{lower:bound:on:the:prob}
\end{align}
Suppose we are able to show that the TV distance between these distributions can be made arbitrarily small, i.e. for any $\epsilon > 0$ we can ensure that,
\begin{align}
\label{eq:tv}
\TV\bigg(\PP_{p_0 p_{Z}}, \frac{1}{2^m} \sum_{i \in [2^m]} \PP_{p_i p_{Z}}\bigg)  \leq \epsilon,
\end{align}
then as a consequence we obtain that,
\begin{align*}
\MoveEqLeft \inf_{\psi} \max\bigg(\PP_{p_0 p_{Z}} (\psi \neq 0), \frac{1}{2^m} \sum_{i \in [2^m]} \PP_{p_i p_{Z}} (\psi = 0)\bigg) \\
& \geq 
  \inf_{\psi} \max\bigg(\PP_{p_0 p_{Z}} (\psi \neq 0),  \PP_{p_0 p_{Z}} (\psi = 0) - \epsilon \bigg)\\
& \geq \frac{1}{2} - \epsilon. 
\end{align*}
Hence we conclude following~\eqref{simple:ineq:markov} that,
\begin{align*}
\inf_{\hat p_{X|Z}} \max_{i \in \{0\}\cup[2^m]} \EE_{p_{i}p_{Z}} \|\hat  p_{X|Z} p_{Z} - p_i p_{Z}\|_1 \geq \frac{\tvlb}{2} - \tvlb \epsilon. 
\end{align*}
It remains to specify the choice of the constant $c$ in our definition of $p_1$. When $\beta > 1$ or when $\holdercon \leq 2/d_X^\beta$ 
we can choose $c = 0$ to obtain the lower bound of $1/36d_X$. Otherwise, we must choose $c$ large enough 
to ensure that $ 2(1 - c)/d_X^\beta \leq \holdercon$, i.e. we choose $c = 1 - (\holdercon d_X^\beta)/2$ to obtain the lower bound of  $\frac{\holdercon^2 d_X^{(2\beta - 1)}}{144}$ as claimed, completing the proof of the theorem.

Finally, we prove the total variation bound in~\eqref{eq:tv}. Note that,
\begin{align*}
\TV\bigg(\PP_{p_0 p_{Z}}, \frac{1}{2^m} \sum_{i \in [2^m]} \PP_{p_i p_{Z}}\bigg)  = \sup_{A \in \Sigma} \bigg|\PP_{p_0 p_{Z}}(A) - \frac{1}{2^m} \sum_{i \in [2^m]} \PP_{p_i p_{Z}}(A)\bigg|,
\end{align*}
where $\Sigma$ is the Borel $\sigma$-field. Let $B$ be the event where at least two points $Z_i$ for $i \in [n]$ belong to the same bin $C_k$ for some $k$. The complement $B^c$ is therefore the event where each point $Z_i$ for $i \in [n]$ falls into its own bin. For any event $A$ we have
\begin{align*}
\bigg|\PP_{p_0 p_{Z}}(A) - \frac{1}{2^m} \sum_{i \in [2^m]} \PP_{p_i p_{Z}}(A)\bigg| & \leq \bigg|\PP_{p_0 p_{Z}}(A \cap B^c) - \frac{1}{2^m} \sum_{i \in [2^m]} \PP_{p_i p_{Z}}(A\cap B^c)\bigg| \\
& + \bigg|\PP_{p_0 p_{Z}}(A \cap B) - \frac{1}{2^m} \sum_{i \in [2^m]} \PP_{p_i p_{Z}}(A\cap B)\bigg| \\
 & \leq \bigg|\PP_{p_0 p_{Z}}(A \cap B^c) - \frac{1}{2^m} \sum_{i \in [2^m]} \PP_{p_i p_{Z}}(A\cap B^c)\bigg| \\
& + \bigg|\PP_{p_0 p_{Z}}(A \cap B)\bigg| + \bigg| \frac{1}{2^m} \sum_{i \in [2^m]} \PP_{p_i p_{Z}}(A\cap B)\bigg| \\
 & \leq \bigg|\PP_{p_0 p_{Z}}(A \cap B^c) - \frac{1}{2^m} \sum_{i \in [2^m]} \PP_{p_i p_{Z}}(A\cap B^c)\bigg| + 2 \PP_{p_0 p_{Z}}(B), 
\end{align*}
where in the last inequality we used the fact that $\PP_{p_0 p_{Z}}(B) =  \frac{1}{2^m} \sum_{i \in [2^m]} \PP_{p_i p_{Z}}(B)$ since the two distributions have the same marginal distribution on $Z$. Now by the definition of our distributions $p_i$ we know that the mixture distribution $\frac{1}{2^m} \sum_{i \in [2^m]} \PP_{p_i p_{Z}}$ assigns the same measure to the set $A \cap B^c$ as the distribution $\PP_{p_0p_{Z}}$, and therefore 
\begin{align*}
\bigg|\PP_{p_0 p_{Z}}(A) - \frac{1}{2^m} \sum_{i \in [2^m]} \PP_{p_i p_{Z}}(A)\bigg| & \leq  2 \PP_{p_0 p_{Z}}(B).
\end{align*}
Due to the definitions of the sets $C_i$ we have that $\PP_{p_0 p_{Z}}(B) = \frac{m^n - m(m-1)\ldots (m-n + 1)}{m^n} = O(\frac{1}{m})$, for a sufficiently large $m$. It follows that  
\begin{align*}
\TV\bigg(\PP_{p_0 p_{Z}}, \frac{1}{2^m} \sum_{i \in [2^m]} \PP_{p_i p_{Z}}\bigg)  \leq O(\frac{1}{m}). 
\end{align*}
Thus going back to \eqref{lower:bound:on:the:prob}, we have
\begin{align*}
\MoveEqLeft \inf_{\psi} \max\bigg(\PP_{p_0 p_{Z}} (\psi \neq 0), \frac{1}{2^m} \sum_{i \in [2^m]} \PP_{p_i p_{Z}} (\psi = 0)\bigg) \\
& \geq 
  \inf_{\psi} \max\bigg(\PP_{p_0 p_{Z}} (\psi \neq 0),  \PP_{p_0 p_{Z}} (\psi = 0) - O(\frac{1}{m})\bigg)\\
& \geq \frac{1}{2} - O\bigg(\frac{1}{m}\bigg). 
\end{align*}
Taking $m$ large enough completes the proof. 

\subsection{Proof of Theorem \ref{upper:bound:holder:thm}}\label{proof:of:upper:bound}

For each $A_{\bar{j}}$, define the corresponding estimate
\begin{align*}
\hat{p}_{X,\bar{j}}(x) := \frac{\sum_{i \in [n]} \mathbbm{1}(Z_i \in A_{\bar{j}}) K(\frac{X_i - x}{h})}{h^{d_X} \sum_{i \in [n]} \mathbbm{1} (Z_i \in A_{\bar{j}})}.
\end{align*}
Using the triangle inequality we have
\begin{align}
& \mathbb{E} \int \int |\hat{p}_{X|Z}(x|z) - p_{X|Z}(x|z)|p_{Z}(z)dx dz \nonumber \\
& \leq \int \int |\tilde{p}_{X|Z}(x|z) - p_{X|Z}(x|z)|p_{Z}(z)dx dz + \mathbb{E} \int \int |\hat{p}_{X|Z}(x|z) - \tilde{p}_{X|Z}(x|z)|p_{Z}(z)dx dz,\label{first:second:term}
\end{align}
where
\begin{align*}
\tilde{p}_{X|Z}(x|z) = \sum_{\bar{j}} \mathbbm{1}\left(z \in A_{\bar{j}} \right) \mathbb{E} \bigg[\frac{\sum_{i \in [n]} \mathbbm{1}(Z_i \in A_{\bar{j}}) K(\frac{X_i - x}{h})}{h^{d_X} \sum_{i \in [n]} \mathbbm{1} (Z_i \in A_{\bar{j}})}\bigg] =  \sum_{\bar{j}} \mathbbm{1}\left(z \in A_{\bar{j}} \right) \mathbb{E}[\hat{p}_{X,\bar{j}}(x)],
\end{align*}
and
\begin{align*}
\hat{p}_{X|Z}(x|z) = \sum_{\bar{j}} \mathbbm{1}\left(z \in A_{\bar{j}} \right) \frac{\sum_{i \in [n]} \mathbbm{1}(Z_i \in A_{\bar{j}}) K(\frac{X_i - x}{h})}{h^{d_X} \sum_{i \in [n]} \mathbbm{1} (Z_i \in A_{\bar{j}})} = \sum_{\bar{j}} \mathbbm{1}\left(z \in A_{\bar{j}} \right)\hat{p}_{X,\bar{j}}(x).
\end{align*}

We proceed to bound the two terms of \eqref{first:second:term} separately.\\
\\
\noindent \textbf{Bounding the first term of \eqref{first:second:term}}:\\

\noindent By Lemma \ref{lemma:expected} in Appendix \ref{lemmas:sec3:appendix} we know that
\begin{align*}
\mathbb{E}[\hat{p}_{X,\bar{j}}(x)] = h^{-d_X} \mathbb{E} \bigg[ K\left( \frac{X-x}{h} \right) \bigg| Z \in A_{\bar{j}} \bigg] (1 - \mathbb{P}(Z \in A_{\bar{j}}^c)^n).
\end{align*}
So we have
\begingroup
\allowdisplaybreaks
\begin{align*}
& \int \int |\tilde{p}_{X|Z}(x|z) - p_{X|Z}(x|z)|p_{Z}(z)dx dz \\
= &  \int \int \bigg | \sum_{\bar{j}} \mathbbm{1}(z \in A_{\bar{j}}) \left( \mathbb{E}[\hat{p}_{X,\bar{j}}(x)]  - p_{X|Z}(x|z) \right) \bigg | p_{Z}(z)dx dz \\
\leq & \sum_{\bar{j}} \int_{A_{\bar{j}}} \int \bigg |  \mathbb{E}[\hat{p}_{X,\bar{j}}(x)]  - p_{X|Z}(x|z) \bigg |p_{Z}(z)dx dz \\
= &  \sum_{\bar{j}} \int_{A_{\bar{j}}} \int  \bigg | h^{-d_X} \mathbb{E}\bigg [ K\left( \frac{X-x}{h} \right) \bigg| Z \in A_{\bar{j}} \bigg] (1 - \mathbb{P}(Z \in A_{\bar{j}}^c)^n)  \\
& \quad \quad \quad \quad \quad - p_{X|Z}(x|z)  (1 - \mathbb{P}(Z \in A_{\bar{j}}^c)^n) - p_{X|Z}(x|z) \mathbb{P}(Z \in A_{\bar{j}}^c)^n \bigg |p_{Z}(z)dx dz \\
\leq &  \sum_{\bar{j}} \int_{A_{\bar{j}}} \int \bigg | (1 - \mathbb{P}(Z \in A_{\bar{j}}^c)^n) \biggl( h^{-d_X} \mathbb{E} \bigg[ K\left( \frac{X-x}{h} \right) \bigg| Z \in A_{\bar{j}} \bigg] - p_{X|Z}(x|z) \biggr) \bigg |  p_{Z}(z)dx dz \\
&  + \sum_{\bar{j}} \int_{A_{\bar{j}}} \int p_{X|Z}(x|z) \mathbb{P}(Z \in A_{\bar{j}}^c)^n  p_{Z}(z)dx dz \\
\leq & \sum_{\bar{j}} \int_{A_{\bar{j}}} \int \bigg  | h^{-d_X} \mathbb{E} \bigg[ K\left( \frac{X-x}{h} \right) \bigg| Z \in A_{\bar{j}} \bigg] - p_{X|Z}(x|z) \bigg |  p_{Z}(z)dx dz \\
&  + \sum_{\bar{j}} \int_{A_{\bar{j}}} \int p_{X|Z}(x|z) \mathbb{P}(Z \in A_{\bar{j}}^c)^n  p_{Z}(z)dx dz.  \stepcounter{equation}\tag{\theequation}\label{second:term:in:first:term:bound}
\end{align*}
\endgroup

We consider the two terms separately. To upper bound the first term, we know from Lemma \ref{lemma:upper} in Appendix \ref{lemmas:sec3:appendix} that 
\begin{align*}
 \bigg| h^{-d_X} \mathbb{E} \bigg[ K\left( \frac{X-x}{h} \right) \bigg| Z \in A_{\bar{j}}\bigg] - p_{X|Z}(x|z \in A_{\bar{j}}) \bigg| \leq C h^\beta,
 \end{align*}
for some constant $C$. Then applying the triangle inequality we have
\begin{align*}
& \sum_{\bar{j}} \int_{A_{\bar{j}}} \int \bigg  | h^{-d_X} \mathbb{E} \bigg[ K\left( \frac{X-x}{h} \right) \bigg| Z \in A_{\bar{j}} \bigg] - p_{X|Z}(x|z) \bigg |  p_{Z}(z)dx dz \\
\leq & \sum_{\bar{j}} \int_{A_{\bar{j}}} \int \bigg( \bigg |p_{X|Z}(x|z) - p_{X|Z}(x|z \in A_{\bar{j}}) \bigg| \\
& \qquad \qquad \qquad + \bigg  | h^{-d_X} \mathbb{E} \bigg[ K\left( \frac{X-x}{h} \right) \bigg| Z \in A_{\bar{j}} \bigg] - p_{X|Z}(x|z \in A_{\bar{j}}) \bigg | \bigg) p_{Z}(z)dx dz \\
\leq & \sum_{\bar{j}} \int_{A_{\bar{j}}} \int \bigg|p_{X|Z}(x|z) - \int_{A_{\bar{j}}} p_{X|Z}(x|z') \frac{p_{Z}(z') }{\mathbb{P} (Z \in A_{\bar{j}})} dz' \bigg| p_{Z}(z)dx dz  + Ch^{\beta} \\
\leq & \sum_{\bar{j}} \int_{A_{\bar{j}}}  \int_{A_{\bar{j}}} \int \bigg|p_{X|Z}(x|z) - p_{X|Z}(x|z')\bigg| dx p_{Z}(z) \frac{p_{Z}(z') }{\mathbb{P} (Z \in A_{\bar{j}})} dz'dz  + Ch^{\beta} \\
= & \sum_{\bar{j}} \int_{A_{\bar{j}}}  \int_{A_{\bar{j}}} \|p_{X|Z=z} - p_{X|Z=\bm{z'}}\|_1 p_{Z}(z) \frac{p_{Z}(z') }{\mathbb{P} (Z \in A_{\bar{j}})} dz'dz  + Ch^{\beta} \\
\leq & \sum_{\bar{j}} \int_{A_{\bar{j}}}  \int_{A_{\bar{j}}} \tvcon \|z - z'\|_1^{\gamma} p_{Z}(z) \frac{p_{Z}(z') }{\mathbb{P} (Z \in A_{\bar{j}})} dz'dz  + Ch^{\beta} \\
\leq & \sum_{\bar{j}} \int_{A_{\bar{j}}}  \int_{A_{\bar{j}}} \tvcon \left(\frac{d_Z}{m}\right)^{\gamma} p_{Z}(z) \frac{p_{Z}(z') }{\mathbb{P} (Z \in A_{\bar{j}})} dz'dz  + Ch^{\beta} \\
= & \sum_{\bar{j}} \int_{A_{\bar{j}}} \tvcon \left(\frac{d_Z}{m}\right)^{\gamma} p_{Z}(z) dz  + Ch^{\beta} \\
= &  \tvcon \left(\frac{d_Z}{m}\right)^{\gamma} + Ch^{\beta},
\end{align*}
where $\tvcon$ and $C$ are constants.

Now we upper bound the second term in equation \eqref{second:term:in:first:term:bound}. Notice that it reduces to:
\begin{align*}
& \sum_{\bar{j}} \int_{A_{\bar{j}}} \int p_{X|Z}(x|z) \mathbb{P}(Z \in A_{\bar{j}}^c)^n  p_{Z}(z)dx dz 
= \sum_{\bar{j}} p_{\bar{j}} (1 - p_{\bar{j}})^n,
\end{align*}
where $p_{\bar{j}} = \mathbb{P}(Z \in A_{\bar{j}})$.

\begin{lemma}\label{lagrange:multipliers:lemma} We have
\begin{align}\label{thm:3.5:prob:eq}
 \sum_{\bar{j}} p_{\bar{j}} (1 - p_{\bar{j}})^n \leq \bigg(1 - \frac{1}{m^{d_Z}}\bigg)^n.
\end{align}
\end{lemma}
The proof of Lemma \ref{lagrange:multipliers:lemma} relies on Lagrange multipliers and is deferred to Appendix \ref{lemmas:sec3:appendix}. Finally we have an upper bound for the entirety of the first term of \eqref{first:second:term}:
\begin{align*}
& \int \int |\tilde{p}_{X|Z}(x|z) - p_{X|Z}(x|z)|p_{Z}(z)dx dz \leq \tvcon \left(\frac{d_Z}{m}\right)^{\gamma} + Ch^{\beta} + \bigg(1 - \frac{1}{m^{d_Z}}\bigg)^n. 
\end{align*}

\noindent \textbf{Bounding the second term of \eqref{first:second:term}}:
\\
\\
Recall that both $\tilde{p}_{X|Z}(x|z)$ and $\hat{p}_{X|Z}(x|z)$ are of the form $ \sum_{\bar{j}} \mathbbm{1}\left(z \in A_{\bar{j}} \right) M$ where the first part  $\sum_{\bar{j}} \mathbbm{1}\left(z \in A_{\bar{j}} \right)$ depends on $z$ while the second part $M$ is independent of $z$. Then we can rewrite the integral as
\begin{align*}
& \mathbb{E} \int \int |\hat{p}_{X|Z}(x|z) - \tilde{p}_{X|Z}(x|z)|p_{Z}(z)dx dz \\
\leq & \mathbb{E} \bigg[ \sum_{\bar{j}}  \int \int \mathbbm{1}\left(z \in A_{\bar{j}} \right)  \bigg| \hat{p}_{X,\bar{j}}(x) -  \mathbb{E}[\hat{p}_{X,\bar{j}}(x)] \bigg| p_{Z}(z)dx dz \bigg] \\
= & \mathbb{E} \bigg[ \sum_{\bar{j}} \int_{A_{\bar{j}}} p_{Z}(z) dz \int  \bigg| \hat{p}_{X,\bar{j}}(x) -  \mathbb{E}[\hat{p}_{X,\bar{j}}(x)] \bigg| dx  \bigg] \\
= & \mathbb{E} \bigg[ \sum_{\bar{j}} \int \mathbb{P}(Z \in A_{\bar{j}})  \bigg| \hat{p}_{X,\bar{j}}(x) -  \mathbb{E}[\hat{p}_{X,\bar{j}}(x)] \bigg| dx  \bigg] \\
= & \sum_{\bar{j}} \int \mathbb{P}(Z \in A_{\bar{j}}) \mathbb{E} \bigg[ \bigg| \hat{p}_{X,\bar{j}}(x) -  \mathbb{E}[\hat{p}_{X,\bar{j}}(x)] \bigg| \bigg] dx.
\end{align*}

We can further bound this term by first bounding the inner expression. By Jensen's inequality we have
\begin{align*}
\mathbb{E} \bigg[ \big| \hat{p}_{X,\bar{j}}(x) -  \mathbb{E}[\hat{p}_{X,\bar{j}}(x)] \big| \bigg] \leq \sqrt{\mathbb{E} \bigg[ \big( \hat{p}_{X,\bar{j}}(x) -  \mathbb{E}[\hat{p}_{X,\bar{j}}(x)] \big)^2 \bigg]} = \sqrt{\var[\hat{p}_{X,\bar{j}}(x)]}.
\end{align*}

But by Lemma \ref{lemma:variance} in Appendix \ref{lemmas:sec3:appendix} we know this variance is upper bounded as,
\begin{align*}
\var[\hat{p}_{X,\bar{j}}(x)] \leq \frac{A}{nh^{d_X} \mathbb{P}(Z \in A_{\bar{j}})} + B \mathbb{P}(Z \in A_{\bar{j}}^c)^n,
\end{align*}
for some constants $A, B$. Notice that this bound is independent of $x$, so substituting back into the second term we have
\begin{align*}
& \sum_{\bar{j}} \int \mathbb{P}(Z \in A_{\bar{j}}) \mathbb{E} \bigg[ \bigg| \hat{p}_{X,\bar{j}}(x) -  \mathbb{E}[\hat{p}_{X,\bar{j}}(x)] \bigg| \bigg] dx \\
\leq & K \sum_{\bar{j}} \int \mathbb{P}(Z \in A_{\bar{j}}) \sqrt{\frac{1}{nh^{d_X} \mathbb{P}(Z \in A_{\bar{j}})} + \mathbb{P}(Z \in A_{\bar{j}}^c)^n} dx \\
= & K \sum_{\bar{j}} \sqrt{\mathbb{P}(Z \in A_{\bar{j}})} \sqrt{\frac{1}{nh^{d_X} } + \mathbb{P}(Z \in A_{\bar{j}}) \mathbb{P}(Z \in A_{\bar{j}}^c)^n},
\end{align*}
where $K$ is a constant. Once again denote $p_{\bar{j}} = \mathbb{P}(Z \in A_{\bar{j}})$. Then by Cauchy-Schwarz we have 
\begin{align*}
K \sum_{\bar{j}} \sqrt{p_{\bar{j}}} \sqrt{\frac{1}{nh^{d_X}} + p_{\bar{j}} (1 - p_{\bar{j}})^n}  & \leq  K \sqrt{\left( \sum_{\bar{j}} p_{\bar{j}} \right)\left(\sum_{\bar{j}} \frac{1}{nh^{d_X}} + \sum_{\bar{j}}  p_{\bar{j}} (1 - p_{\bar{j}})^n \right)} \\
& = K \sqrt{\frac{m^{d_Z}}{nh^{d_X}} + \sum_{\bar{j}}  p_{\bar{j}} (1 - p_{\bar{j}})^n} \\
& \leq K \sqrt{\frac{m^{d_Z}}{nh^{d_X}}} + K \bigg(1 - \frac{1}{m^{d_Z}}\bigg)^{\frac{n}{2}},
\end{align*}
where the last step follows from \eqref{thm:3.5:prob:eq}, since we already proved that $ \sum_{\bar{j}}  p_{\bar{j}} \big(1 - p_{\bar{j}}\big)^n \leq (1 - \frac{1}{m^{d_Z}})^n$.

So we have shown that the second term is upper bounded as
\begin{align*}
\mathbb{E} \int \int |\hat{p}_{X|Z}(x|z) - \tilde{p}_{X|Z}(x|z)|p_{Z}(z)dx dz \leq K \sqrt{\frac{m^{d_Z}}{nh^{d_X}}} + K \bigg(1 - \frac{1}{m^{d_Z}}\bigg)^{\frac{n}{2}}.
\end{align*}\\

\noindent \textbf{Combining the terms}:\\\\
Combining the bounds for the two terms, we have found an upper bound to the loss function:
\begin{align*}
\MoveEqLeft \mathbb{E} \int \int |\hat{p}_{X|Z}(x|z) - p_{X|Z}(x|z)|p_{Z}(z)dx dz \\
& \leq \tvcon \left(\frac{d_Z}{m}\right)^{\gamma} + Ch^{\beta} + \bigg(1 - \frac{1}{m^{d_Z}}\bigg)^n + K \sqrt{\frac{m^{d_Z}}{nh^{d_X}}} + K \bigg(1 - \frac{1}{m^{d_Z}}\bigg)^{\frac{n}{2}},
\end{align*}
where $\tvcon$, $C$ and $K$ are constants. This yields the optimal parameter values $h \asymp n^{\frac{-1}{d_X + d_Z \beta \gamma^{-1} + 2\beta}}$ and $m \asymp n^{\frac{\beta}{d_X + d_Z \beta \gamma^{-1} + 2\beta}}$, and a rate of $n^{\frac{-1}{\beta^{-1} d_X + \gamma^{-1} d_Z + 2}}$.
Notice that by our selection of $h, m$, we have $(1 - \frac{1}{m^{d_Z}})^n \leq (1 - \frac{1}{m^{d_Z}})^{\frac{n}{2}} \leq \exp(\frac{-n}{2m^{d_Z}})$, both of which are negligible compared to other terms for sufficiently large $n$, and thus can be ignored.


\subsection{Proof of Theorem \ref{lower:bound:holder:thm}}
\label{sec:lower}
We first choose a ``bump'' function $h$ supported on $[0,1]$
which is an infinitely differentiable function, and satisfies the conditions that $\int h(x)dx = 0$, $\int h^2(x)dx = 1$, and for which $\int |h(x)|dx$ is a non-zero constant.

We construct a collection of densities by setting the conditional distributions $p_{X|Z}$ to be the uniform density perturbed by bumps of an appropriate resolution. Formally, 
we bin $[0,1]$ into $m$ bins for $Z$ and $r$ bins for $X$. We then define the following conditional density functions:
\begin{align*}
p^{\Delta}_{X|Z}(x|z) = 1 + \sum_{\bar{i}} \sum_{\bar{j}} \Delta_{\bar{i},\bar{j}} \prod_{k \in [d_X]} h_{i_k,r}(x_k) \prod_{k \in [d_Z]} h_{j_k,m}(z_k),
\end{align*}
where we recall the shorthands $\bar{i} \in [r]^{d_X}, \bar{j} \in [m]^{d_Z}$ ($r, m$ are integers chosen later), and $\Delta_{\bar{i}, \bar{j}} \in \{ \pm 1 \}$. 
We further define
\begin{align*}
h_{i_k,r}(x_k) & = \rho \sqrt{r} h(rx_k - i_k + 1),\\
h_{j_k,m}(z_k) & = \rho \sqrt{m} h(mz_k - j_k + 1),
\end{align*}
where $\rho$ is a positive constant which we will choose appropriately. The support of $h_{i_k,r}(x_k)$ is $x_k \in [\frac{i_k-1}{r}, \frac{i_k}{r}]$, and the support of $h_{j_k,m}(z_k)$ is $z_k \in [\frac{j_k-1}{m}, \frac{j_k}{m}]$. The supports of these bumps are disjoint for different values of $i_k$ or $j_k$.  

The following lemma develops some important properties of the perturbed densities $p^{\Delta}$.
\begin{lemma}
\label{lem:pert}
\begin{enumerate}
\item Suppose we ensure that, 
\begin{align}
\label{eqn:c1}
 \rho^d r^{d_X/2} m^{d_Z/2} \|h\|_{\infty}^d \leq \frac{1}{2},
\end{align}
then $p^\Delta$ is a valid density.
\item Suppose we ensure that, 
\begin{align}
2 \rho^d r^{d_X/2} m^{d_Z/2} r^{\beta}  \|h\|_{\infty}^{d_Z} \bigg( \bigg[2 \prod_{k \in [d_X]} \|h^{(\alpha_k)}\|_{\infty}\bigg] \vee \bigg[\sqrt{d_X}   \prod_{k \in [d_X]} \left[ \|h^{(\alpha_k+1)}\|_{\infty} \vee \|h^{(\alpha_k)}\|_{\infty}  \right]\bigg]\bigg) \leq \holdercon,
\end{align} 
then $p^\Delta$ satisfies the H\"{o}lder smoothness condition.
\item Finally, if we ensure that, 
\begin{align}
2 \rho^d r^{\frac{d_X}{2}} m^{\frac{d_Z}{2}} m^{\gamma} \|h\|_{\infty}^{d_Z - 1}  \left( 2 \|h\|_{\infty} \vee \|h'\|_{\infty} \right) \leq \tvcon,
\end{align}
then $p^\Delta$ satisfies the TV smoothness condition.
\end{enumerate}
\end{lemma}

\noindent Our lower bound will follow as a consequence of Fano's inequality~\cite{Yu1997fano}. In order to apply Fano's inequality it
will be useful to bound the KL divergence between a pair of densities $p^\Delta$ and $p^{\Delta'}$, and to show that
we can construct a collection of sufficiently large cardinality which are well-separated in the loss function~\eqref{loss:function:def}.
\begin{lemma}
\label{lem:kl}
\begin{enumerate}
\item Suppose that the condition in~\eqref{eqn:c1} holds. For a given pairs of densities $p^\Delta$ and $p^{\Delta'}$ the KL divergence can be bounded as,
\begin{align*}
\text{KL}(p^\Delta, p^{\Delta'}) \leq  8 \|h\|_{\infty} \rho^{2d} r^{d_X} m^{d_Z}.
\end{align*}
\item There is a subset $T$ of densities $p^{\Delta}$ such that, $|T| \geq 2^{r^{d_X} m^{d_Z}/8}$, and furthermore for any pair of densities $p^\Delta$ and $p^{\Delta'}$ in $T$ there is an absolute constant $C > 0$ such that,
\begin{align*}
\int \int |p^{\Delta}_{X|Z}(x|z) - p^{\Delta'}_{X|Z}(x|z)| p_Z(z) dx dz \geq \frac{1}{4} \|h\|_1^d \denslower \rho^d  r^{d_X/2} m^{d_Z/2}.
\end{align*}
\end{enumerate}
\end{lemma}
Ignoring constants in the remainder of the proof we now describe our choice of $(\rho,r,m)$. We select $\rho^d \asymp 1/\sqrt{n}$, 
$r \asymp n^{\beta^{-1}/(d_X/\beta + d_Z/\gamma + 2)}$ and $m \asymp n^{\gamma^{-1}/(d_X/\beta + d_Z/\gamma + 2)}$ (with appropriately small constants) and observe that each of the conditions of Lemma~\ref{lem:pert} are satisfied.

The proof of the theorem follows from a straightforward application of Fano's inequality (see for instance \cite[Theorem 2.7]{tsybakov09introduction}). 
In rough terms, we apply Fano's inequality to the collection of distributions $T$.
Provided that we can show that the average pairwise KL divergence between the n-sample product distributions is at most some small constant times $\log |T|$ we obtain a lower bound on the loss~\eqref{loss:function:def} of any estimator
which scales as $\rho^d  r^{d_X/2} m^{d_Z/2} \asymp n^{-1/(d_X/\beta + d_Z/\gamma + 1)}$. 
We note that,
\begin{align*}
\text{KL}((p^\Delta)^n, (p^{\Delta'})^n) = n\text{KL}(p^\Delta, p^{\Delta'}) \lesssim n \rho^{2d} r^{d_X} m^{d_Z} \lesssim r^{d_X} m^{d_Z} \lesssim \log |T|,
\end{align*}
as desired, completing the proof of the theorem.

\subsubsection{Proof of Lemma~\ref{lem:pert}}
We prove each of the three claims in turn.

\paragraph{Proof of Claim (1): }
We now verify that $p^{\Delta}_{X|Z}(x|z)$ is a density function. Consider the integral
\begin{align*}
& \int \int \left( 1 + \sum_{\bar{i}} \sum_{\bar{j}} \Delta_{\bar{i},\bar{j}} \prod_{k \in [d_X]} h_{i_k,r}(x_k) \prod_{k \in [d_Z]} h_{j_k,m}(z_k) \right) dx dz \\
= & 1 + \sum_{\bar{i}} \sum_{\bar{j}} \Delta_{\bar{i},\bar{j}} \int h_{i_1,r}(x_1) dx_1 ... \int h_{j_{d_Z},m}(z_{d_Z}) dz_{d_Z} \\
= & 1 + \sum_{\bar{i}} \sum_{\bar{j}} \Delta_{\bar{i},\bar{j}} \left( \rho r^{-\frac{1}{2}} \int h(u)du \right) ... \left( \rho m^{-\frac{1}{2}} \int h(u)du \right) \\
= & 1.
\end{align*}
And under the additional assumption 
\begin{align*}
|  \Delta_{\bar{i},\bar{j}}  \prod_{k \in [d_X]} h_{i_k,r}(x_k) \prod_{k \in [d_Z]} h_{j_k,m}(z_k) | \leq \rho^d r^{d_X/2} m^{d_Z/2} \|h\|_{\infty}^d \leq \frac{1}{2},
\end{align*}
the function $p^{\Delta}_{X|Z}(x|z)$ is always positive. So it is indeed a density function.

\paragraph{Proof of Claim (2): } Now we verify that $p^{\Delta}_{X|Z}(x|z)$ indeed satisfies the H\"older smoothness assumption. Since the $L_2$ norm is always smaller than or equal to the $L_1$ norm, it suffices to show
\begin{align*}
|D^{\alpha}p^{\Delta}_{X|Z}(x|z) - D^{\alpha}p^{\Delta}_{X|Z}(x'|z)| \leq \holdercon \|x - x'\|_2^{\beta - \ell}.
\end{align*}

$p^{\Delta}_{X|Z}(x|z)$ is infinitely differentiable since $h$ is infinitely differentiable, and so it is $\ell = \floor{\beta}$ times differentiable. Now we want to show the second requirement. Without loss of generality let $(\alpha_1,...,\alpha_{d_X})$ be fixed, and suppose we are given arbitrary $x, x', z$.


Since $z$ is fixed and $\prod_{k \in [d_Z]} h_{j_k,m}(z_k)$ have disjoint support, the only non-zero one is $\prod_{k \in [d_Z]} h_{j_k^*,m}(z_k)$ for the bins $\bar{j}^* = (j_1^*, ..., j_{d_Z}^*)$.

So we have
\begin{align*}
&|D^{\alpha}p^{\Delta}_{X|Z}(x|z) - D^{\alpha}p^{\Delta}_{X|Z}(x'|z)|  \\
= & \bigg | \sum_{\bar{i}} \sum_{\bar{j}} \Delta_{\bar{i},\bar{j}} \left(  \prod_{k \in [d_X]} h_{i_k,r}^{(\alpha_k)}(x_k) \prod_{k \in [d_Z]} h_{j_k,m}(z_k) -  \prod_{k \in [d_X]} h_{i_k,r}^{(\alpha_k)}(x_k') \prod_{k \in [d_Z]} h_{j_k,m}(z_k)\right) \bigg | \\
= & \rho^d r^{d_X/2} m^{d_Z/2} r^{\ell}  \bigg | \sum_{\bar{i}} \Delta_{\bar{i},\bar{j}^*} \bigg( \prod_{k \in [d_X]} h^{(\alpha_k)}(rx_k-i_k+1) \prod_{k \in [d_Z]} h(mz_k-j_k^*+1) \\
& \qquad \qquad \qquad \qquad \qquad \qquad \quad - \prod_{k \in [d_X]} h^{(\alpha_k)}(rx_k'-i_k+1) \prod_{k \in [d_Z]} h(mz_k-j_k^*+1) \bigg) \bigg | \\
\leq &  \rho^d r^{d_X/2} m^{d_Z/2} r^{\ell}  \|h\|_{\infty}^{d_Z} \bigg | \sum_{\bar{i}} \Delta_{\bar{i},\bar{j}^*} \left( \prod_{k \in [d_X]} h^{(\alpha_k)}(rx_k-i_k+1)  - \prod_{k \in [d_X]} h^{(\alpha_k)}(rx_k'-i_k+1)  \right) \bigg |.
\end{align*}

Notice that in the above summation, there are at most two non-zero terms, as $\prod_{k \in [d_X]} h^{(\alpha_k)}(rx_k-i_k+1)$ have disjoint supports. Let $\bar{a} = (a_1, ..., a_{d_X})$ be such that $\forall k \in [d_x], x_k \in [\frac{a_k-1}{r}, \frac{a_k}{r}]$, and let $\bar{b} = (b_1, ..., b_{d_Z})$ be such that $\forall k \in [d_X], x_k' \in [\frac{b_k-1}{r}, \frac{b_k}{r}]$, which correspond to the two non-zero terms respectively. Then we have
\begin{align*}
& |D^{\alpha}p^{\Delta}_{X|Z}(x|z) - D^{\alpha}p^{\Delta}_{X|Z}(x'|z)| \\
\leq & \rho^d r^{d_X/2} m^{d_Z/2} r^{\ell}  \|h\|_{\infty}^{d_Z} \bigg[ |\prod_{k \in [d_X]} h^{(\alpha_k)}(rx_k-a_k+1)  - \prod_{k \in [d_X]} h^{(\alpha_k)}(rx_k'-a_k+1)| \\
& \qquad \qquad \qquad \qquad \qquad+ |\prod_{k \in [d_X]} h^{(\alpha_k)}(rx_k-b_k+1)  - \prod_{k \in [d_X]} h^{(\alpha_k)}(rx_k'-b_k+1)| \bigg].
\end{align*}

We can further bound the term within the square brackets. We will find two upper bounds and use the minimum between the two. Firstly we have 
\begin{align*}
|& \prod_{k \in [d_X]} h^{(\alpha_k)}(rx_k-a_k+1)  - \prod_{k \in [d_X]} h^{(\alpha_k)}(rx_k'-a_k+1)|  \\
& \qquad + |\prod_{k \in [d_X]} h^{(\alpha_k)}(rx_k-b_k+1)  - \prod_{k \in [d_X]} h^{(\alpha_k)}(rx_k'-b_k+1)| \\
\leq &4 \prod_{k \in [d_X]} \|h^{(\alpha_k)}\|_{\infty} := \mu_1.
\end{align*}

Secondly, using the identity $|f(x) - f(x')| \leq \sup_{y} \|\nabla f(y)\|_2 \|x-x'\|_2$, where $\nabla$ is the gradient and we take $f(x) = \prod_{k \in [d_X]} h^{(\alpha_k)}(rx_k - i_k + 1)$, we have another upper bound. Here we have
\begin{align*}
& \sup_{y} \|\nabla f(y)\|_2 \\
= & \sup_{y} \sqrt{\sum_{l \in [d_X]} \left(rh^{(\alpha_l + 1)}(rx_l - i_l + 1) \prod_{k \in [d_X], k \neq l} h^{(\alpha_k)}(rx_k - i_k + 1) \right)^2} \\
\leq & r  \sqrt{\sum_{l \in [d_X]} \left(\|h^{(\alpha_l+1)}\|_{\infty} \prod_{k \in [d_X], k \neq l} \|h^{(\alpha_k)}\|_{\infty} \right)^2} \\
\leq & r \sqrt{d_X}   \prod_{k \in [d_X]} \left[ \|h^{(\alpha_k+1)}\|_{\infty} \vee \|h^{(\alpha_k)}\|_{\infty}  \right].
\end{align*}
This identity gives us the upper bound
\begin{align*}
|& \prod_{k \in [d_X]} h^{(\alpha_k)}(rx_k-a_k+1)  - \prod_{k \in [d_X]} h^{(\alpha_k)}(rx_k'-a_k+1)| \\
& \qquad + |\prod_{k \in [d_X]} h^{(\alpha_k)}(rx_k-b_k+1)  - \prod_{k \in [d_X]} h^{(\alpha_k)}(rx_k'-b_k+1)| \\
\leq & 2 r \sqrt{d_X}   \prod_{k \in [d_X]} \left[ \|h^{(\alpha_k+1)}\|_{\infty} \vee \|h^{(\alpha_k)}\|_{\infty}  \right] \|x-x'\|_2 := \mu_2 r \|x-x'\|_2.
\end{align*}
Taking the minimum of these two upper bounds gives a tighter upper bound. Let $\wedge$ denote the minimum between two terms and $\vee$ the maximum. Using the properties $(ab \wedge cd) \leq (a \vee c)(b \wedge d), a,b,c,d > 0$ and $(1 \wedge u) \leq u^\alpha$ for $u > 0, 0 < \alpha \leq 1$,  we have
\begin{align}\label{proof:lower:holder:eq}
& |D^{\alpha}p^{\Delta}_{X|Z}(x|z) - D^{\alpha}p^{\Delta}_{X|Z}(x'|z)| \nonumber \\
\leq & \rho^d r^{d_X/2} m^{d_Z/2} r^{\ell}  \|h\|_{\infty}^{d_Z} \big[ \mu_1 \wedge  (r \mu_2 \|x-x'\|_2) \big] \nonumber \\
\leq & \rho^d r^{d_X/2} m^{d_Z/2} r^{\ell}  \|h\|_{\infty}^{d_Z} (\mu_1 \vee \mu_2) (1 \wedge r \|x-x'\|_2) \nonumber \\
\leq & \rho^d r^{d_X/2} m^{d_Z/2} r^{\ell}  \|h\|_{\infty}^{d_Z} (\mu_1 \vee \mu_2) r^{\beta - \ell} \|x-x'\|_2^{\beta - \ell} \nonumber \\
\leq &  \holdercon\|x-x'\|_2^{\beta - \ell},
\end{align}
provided we ensure that $\rho^d r^{d_X/2} m^{d_Z/2} r^{\beta} (\mu_1 \vee \mu_2) \leq \holdercon$, which is indeed the case. So $p^{\Delta}_{X|Z}(x|z)$ satisfies the H\"older smoothness condition.

\paragraph{Proof of Claim (3): } Now we show that $p^{\Delta}_{X|Z}(x|z)$ also satisfies the TV smoothness assumption. We have
\begin{align}
\label{eqn:orig}
& \int |p^{\Delta}_{X|Z}(x|z) - p^{\Delta}_{X|Z}(x|z')|dx \leq  \int \sum_{\bar{i}} \sum_{\bar{j}}  \bigg | \prod_{k \in [d_X]} h_{i_k,r}(x_k) \bigg | \bigg |\prod_{k \in [d_Z]} h_{j_k,m}(z_k) - \prod_{k \in [d_Z]} h_{j_k,m}(z_k') \bigg | dx.
\end{align}
Recall that $\prod_{k \in [d_Z]} h_{j_k,m}(z_k)$ have disjoint supports, so there are at most two non-zero terms within the summation. Suppose $\forall k \in [d_Z], z_k \in [\frac{j_k^*-1}{m}, \frac{j_k^*}{m}]$ for some specific $j_k^*$ while $\forall k \in [d_Z], z_k' \in [\frac{j_k'^*-1}{m}, \frac{j_k'^*}{m}]$ for some specific $j_k'^*$, corresponding to the two non-zero terms. Then
\begin{align*}
& \sum_{\bar{j}} \bigg | \prod_{k \in [d_Z]} h_{j_k,m}(z_k) - \prod_{k \in [d_Z]} h_{j_k,m}(z_k') \bigg | \\
\leq & \bigg| \prod_{k \in [d_Z]} h_{j_k^*,m}(z_k) - \prod_{k \in [d_Z]} h_{j_k^*,m}(z_k') \bigg | + \bigg| \prod_{k \in [d_Z]} h_{j_k'^*,m}(z_k) - \prod_{k \in [d_Z]} h_{j_k'^*,m}(z_k') \bigg |.
\end{align*}
We upper bound the first term and note that an identical upper bound holds for the second term. 
Using a similar approach to how we showed H\"older smoothness in \eqref{proof:lower:holder:eq}, and by telescoping, we have
\begin{align*}
& \bigg| \prod_{k \in [d_Z]} h_{j_k^*,m}(z_k) - \prod_{k \in [d_Z]} h_{j_k^*,m}(z_k') \bigg | \\
\leq & \sum_{k \in [d_Z]} (\sqrt{m}\rho)^{d_Z -1} \|h\|_{\infty}^{d_Z - 1} |h_{j_k^*,m}(z_k) - h_{j_k^*,m}(z_k')| \\
\leq & \sum_{k \in [d_Z]} (\sqrt{m}\rho)^{d_Z -1} \|h\|_{\infty}^{d_Z - 1} \rho \sqrt{m} \left( 2 \|h\|_{\infty} \wedge \|h'\|_{\infty} m |z_k - z_k'| \right) \\
\leq & \sum_{k \in [d_Z]} (\sqrt{m}\rho)^{d_Z -1} \|h\|_{\infty}^{d_Z - 1} \rho \sqrt{m} \left( 2 \|h\|_{\infty} \vee \|h'\|_{\infty} \right) \left( 1 \wedge m |z_k - z_k'| \right) \\
\leq & \sum_{k \in [d_Z]} (\sqrt{m}\rho)^{d_Z -1} \|h\|_{\infty}^{d_Z - 1} \rho \sqrt{m} \left( 2 \|h\|_{\infty} \vee \|h'\|_{\infty} \right) m^{\gamma} |z_k - z_k'|^{\gamma} \\
\leq & \rho^{d_Z} m^{\frac{d_Z}{2}} m^{\gamma} \|h\|_{\infty}^{d_Z - 1}  \left( 2 \|h\|_{\infty} \vee \|h'\|_{\infty} \right)  \|z - z'\|_1^{\gamma}.
\end{align*}

Substituting this result back in~\eqref{eqn:orig} gives
\begin{align*}
& \int |p^{\Delta}_{X|Z}(x|z) - p^{\Delta}_{X|Z}(x|z')|dx \\
= & \sum_{\bar{j}} \bigg |\prod_{k \in [d_Z]} h_{j_k,m}(z_k) - \prod_{k \in [d_Z]} h_{j_k,m}(z_k') \bigg | \sum_{\bar{i}} \int_0^1 \bigg | \prod_{k \in [d_X]} h_{i_k,r}(x_k) \bigg | dx \\
\leq & 2 \rho^{d_Z} m^{\frac{d_Z}{2}} m^{\gamma} \|h\|_{\infty}^{d_Z - 1}  \left( 2 \|h\|_{\infty} \vee \|h'\|_{\infty} \right)  \|z - z'\|_1^{\gamma}  \sum_{\bar{i}} \rho^{d_X} r^{\frac{-d_X}{2}}\prod_{k \in [d_X]} \left( \int_0^1 |h(u)| du \right)\\
= & 2 \rho^d r^{\frac{d_X}{2}} m^{\frac{d_Z}{2}} m^{\gamma} \|h\|_{\infty}^{d_Z - 1}  \left( 2 \|h\|_{\infty} \vee \|h'\|_{\infty} \right)  \|z - z'\|_1^{\gamma} \\
\leq & \tvcon \|z - z'\|_1^{\gamma}
\end{align*}
provided we ensure that $2 \rho^d r^{\frac{d_X}{2}} m^{\frac{d_Z}{2}} m^{\gamma} \|h\|_{\infty}^{d_Z - 1}  \left( 2 \|h\|_{\infty} \vee \|h'\|_{\infty} \right) \leq \tvcon$. This is indeed the case, and we obtain that $p^{\Delta}_{X|Z}(x|z)$ indeed satisfies the TV smoothness assumption.

\subsubsection{Proof of Lemma~\ref{lem:kl}}
We prove each of the two claims in turn.

\paragraph{Proof of Claim (1): } Recall that in constructing $p^\Delta$ and $p^{\Delta'}$ we do not perturb the marginal distribution over $Z$. As a consequence the KL divergence between $p^{\Delta}$ and $p^{\Delta'}$ can 
be written as: 
\begin{align*}
d_{\text{KL}}(p^{\Delta}, p^{\Delta'}) = \mathbb{E}_z d_{\text{KL}}(p^{\Delta}_{X|Z}(x | z), p^{\Delta'}_{X|Z}(x|z)),
\end{align*}
and we focus on upper bounding the KL divergence between the conditional densities. 
\begin{align*}
d_{\text{KL}}(p^{\Delta}_{X|Z}(x|z), p^{\Delta'}_{X|Z}(x|z)) \leq & d_{\chi^2}(p^{\Delta}_{X|Z}(x|z), p^{\Delta'}_{X|Z}(x|z)) \\
= & \int p^{\Delta'}_{X|Z}(x|z) \left( \frac{p^{\Delta}_{X|Z}(x|z)}{p^{\Delta'}_{X|Z}(x|z)} - 1 \right)^2 dx \\
= & \int \frac{(p^{\Delta}_{X|Z}(x|z) - p^{\Delta'}_{X|Z}(x|z))^2}{p^{\Delta'}_{X|Z}(x|z)} dx.
\end{align*}
Recall that by the condition in~\eqref{eqn:c1} we have that 
$|  \Delta_{\bar{i},\bar{j}}  \prod_{k \in [d_X]} h_{i_k,r}(x_k) \prod_{k \in [d_Z]} h_{j_k,m}(z_k) | \leq \rho^d r^{d_X/2} m^{d_Z/2} \|h\|_{\infty}^d \leq \frac{1}{2}$ which implies $p^{\Delta'}_{X|Z}(x|z) \geq \frac{1}{2}$. So we have
\begin{align*}
& d_{\text{KL}}(p^{\Delta}_{X|Z}(x|z), p^{\Delta'}_{X|Z}(x|z)) \\
\leq & 2 \int \left( \sum_{\bar{i}} \sum_{\bar{j}} (\Delta_{\bar{i},\bar{j}} - \Delta_{\bar{i},\bar{j}}') \prod_{k \in [d_X]} h_{i_k,r}(x_k) \prod_{k \in [d_Z]} h_{j_k,m}(z_k) \right)^2 dx \\
\stackrel{\text{(i)}}{=} & 2  \sum_{\bar{i}} \sum_{\bar{j}}  (\Delta_{\bar{i},\bar{j}} - \Delta_{\bar{i},\bar{j}}')^2  \prod_{k \in [d_Z]} h^2_{j_k,m}(z_k) \int \prod_{k \in [d_X]} h^2_{i_k,r}(x_k) dx, \\
\leq & 8 \rho^{2d_X} r^{d_X}  \sum_{\bar{j}} \prod_{k \in [d_Z]} h^2_{j_k,m}(z_k),
\end{align*}
where for (i) we note that the cross terms in expanding the square correspond to disjoint bumps and are 0. 
As a consequence we obtain that,
\begin{align*}
d_{\text{KL}}(p^{\Delta}, p^{\Delta'}) &\leq 8 \rho^{2d_X} r^{d_X}  \int \sum_{\bar{j}}  \prod_{k \in [d_Z]} h^2_{j_k,m}(z_k) p_Z(z) dz \\
& \leq 8 \rho^{2d_X} r^{d_X} (\rho \sqrt{m})^{2 d_Z} \| h \|_{\infty} \times \\
&~~~~~~~~~~~~~~~~~~\sum_{\bar{j}} \bigg[\prod_{k \in [d_Z]} \mathbbm{1}\bigg(z_k \in \big[\frac{\bar{j}_k-1}{m}, \frac{\bar{j}_k}{m}\big]\bigg) \bigg] p_Z \bigg( \big[\frac{\bar{j}_1-1}{m}, \frac{\bar{j}_1}{m}\big] \times 
\cdots \times \big[\frac{\bar{j}_{d_Z}-1}{m}, \frac{\bar{j}_{d_Z}}{m}\big]\bigg)\\
& = 8\| h \|_{\infty} \rho^{2d} r^{d_X} m^{d_Z} .
\end{align*}



\paragraph{Proof of Claim (2): } 
Given that the marginal density of $Z$ is lower bounded i.e. $p_{Z}(z) \geq \denslower > 0$, it suffices to instead ensure that 
for some absolute constant $C > 0$ we have that,
\begin{align*}
\int \int |p^{\Delta}_{X|Z}(x|z) - p^{\Delta'}_{X|Z}(x|z)|dxdz \geq C \rho^d  r^{d_X/2} m^{d_Z/2}.
\end{align*}
Using the Varshamov-Gilbert construction \cite[Lemma 2.9]{tsybakov09introduction} we know that there exist $N = 2^{r^{d_X} m^{d_Z}/8}$ vectors $\Delta$ on the hypercube $\{\pm 1 \}^{r^{d_X} m^{d_Z}}$ such that $d_H(\Delta, \Delta') \geq r^{d_X} m^{d_Z}/8$ for each $\Delta, \Delta'$ in that set, where $d_H(\Delta, \Delta') = \frac{1}{2} \sum_{\bar{i}} \sum_{\bar{j}} |\Delta_{\bar{i},\bar{j}} - \Delta'_{\bar{i},\bar{j}}|$ is the Hamming distance. Then
\begin{align*}
& \int \int |p^{\Delta}_{X|Z}(x|z) - p^{\Delta'}_{X|Z}(x|z)|dxdz \\
= & \sum_{\bar{i}} \sum_{\bar{j}} |\Delta_{\bar{i},\bar{j}} - \Delta'_{\bar{i},\bar{j}}| \int \int \bigg |\prod_{k \in [d_X]} h_{i_k,r}(x_k) \prod_{k \in [d_Z]} h_{j_k,m}(z_k) \bigg| dxdz \\
\geq &  \frac{r^{d_X} m^{d_Z}}{4} \frac{\rho^d}{\sqrt{r^{d_X} m^{d_Z}}} \left(\int_0^1 |h(u)|du\right)^d \\
\geq & C \rho^d \sqrt{r^{d_X} m^{d_Z}}
\end{align*}
as claimed, where $C = \|h\|_1^d/4$.

%
%
%
%

\subsection{Proof of Theorem~\ref{thm:yatracos}}
We denote by $\Delta_1, \Delta_2$ the following,
\begin{align*}
\Delta_1 &= \sup_{A \in \mathcal{A}} \left|\mathbb{P}_n(A) - \int_A p\right| , \\
\Delta_2 &= \sup_{A \in \mathcal{A}} \sup_{j \in [N]} \left|\int_A\widetilde{f}_j - \frac{1}{n} \sum_{i=1}^n \int_{A^{Z_i}} \widehat{f}_j(x | Z_i) dx\right|. 
\end{align*}
The following lemma is a simple consequence of Hoeffding's inequality, and gives high-probability upper bounds on the above quantities:
\begin{lemma}
\label{lem:hoef}
With probability at least $1 - \delta$,
\begin{align*}
\Delta_1 &\leq \sqrt{\frac{\log (2N/\delta)}{n}}\\
\Delta_2 &\leq \sqrt{\frac{3 \log (2N/\delta)}{2n}}.
\end{align*}
\end{lemma}
\noindent Taking this result as given we complete the proof, before returning to prove it.
Let us denote by $\widecheck{f}$ the minimizer $\argmin_{\widehat{f}_j: j \in [N]} \int_z \|\widehat{f}_{j} - p(x|z)\|_1 p_Z(z) dz$, then we can write:
\begin{align*}
\int_z \|\psi(x|z) - p(x|z)\|_1 p_Z(z) dz \leq \underbrace{\int_z \|\psi(x|z) - \widecheck{f}(x|z)\|_1 p_Z(z) dz}_T + \int_z \|\widecheck{f}(x|z)- p(x|z)\|_1 p_Z(z) dz.
\end{align*}
Abusing notation slightly in the remainder of the proof we identify $\widecheck{f}$ and $\psi$ with their corresponding oracle joint densities $\widecheck{f}(x|z)p_Z(z)$ and $\psi(x|z)p_Z(z)$.
We note that,
\begin{align*}
T &\leq 2 \sup_{A \in \mathcal{A}} \left| \int_A \psi - \int_A \widecheck{f} \right| \\
&\leq  2 \sup_{A \in \mathcal{A}} \left[|\int_A \psi - \mathbb{P}_n(A) | + | \int_A \widecheck{f} - \mathbb{P}_n(A)| \right]\\
&\leq 2 \sup_{A \in \mathcal{A}}  \left[ \left| \frac{1}{n} \sum_{i=1}^n \int_{A^{Z_i}} \psi(x | Z_i) dx - \mathbb{P}_n(A) \right| + \left|  \frac{1}{n} \sum_{i=1}^n \int_{A^{Z_i}} \psi(x | Z_i) dx  - \int_A \psi \right| \right.\\
&\left. +  \left| \frac{1}{n} \sum_{i=1}^n \int_{A^{Z_i}} \widecheck{f}(x | Z_i) dx - \mathbb{P}_n(A) \right| +  \left|  \frac{1}{n} \sum_{i=1}^n \int_{A^{Z_i}} \widecheck{f}(x | Z_i) dx  - \int_A \widecheck{f} \right|\right] \\
&\leq 4 \Delta_2 +  4 \sup_{A \in \mathcal{A}}  \left|   \frac{1}{n} \sum_{i=1}^n \int_{A^{Z_i}} \widecheck{f}(x | Z_i) dx - \mathbb{P}_n(A) \right|,
\end{align*}
where in the final inequality we use the definition of the minimum distance estimator, and of $\Delta_2$. We then note that,
\begin{align*}
4 \sup_{A \in \mathcal{A}}  \left|  \frac{1}{n} \sum_{i=1}^n \int_{A^{Z_i}} \widecheck{f}(x | Z_i) dx - \mathbb{P}_n(A) \right| &\leq 4 \Delta_1 + 4 \sup_{A \in \mathcal{A}} \left|  \frac{1}{n} \sum_{i=1}^n \int_{A^{Z_i}} \widecheck{f}(x | Z_i) dx - \int_A p \right| \\
&\leq 4\Delta_1 + 4\Delta_2 + 4 \sup_{A \in \mathcal{A}} \left| \widecheck{f}(A)- \int_A p \right| \\
&\leq 4 \Delta_1 + 4\Delta_2 + 2  \int_z \|\widecheck{f}(x|z)- p(x|z)\|_1 p_Z(z) dz,
\end{align*}
and putting these bounds together with the bounds in Lemma~\ref{lem:hoef} we obtain our claimed result.

\paragraph{Proof of Lemma~\ref{lem:hoef}: } We note that for fixed $A \in \mathcal{A}$ (and a fixed index $j \in [N]$) we are simply bounding the deviation of a sum of bounded (by 1) random variables from their mean. 
This is straightforward for the terms appearing in the definition of $\Delta_1$. For the terms appearing in $\Delta_2$ we note that,
\begin{align*}
\widetilde{f}_j(A) =  \int_{A} \widetilde{f}(x|z) p_Z(z) dx dz = \mathbb{E}_Z \left[ \int_{A^Z} f(x|Z) dx\right].
\end{align*}
The result then follows 
by combining the Hoeffding bound with the union bound, noting that $\mathcal{A}$ has cardinality at most $N^2$.

\subsection{Proof of Theorem~\ref{thm:adaptive}}
\label{sec:adaptive_proof}
The proof is a straightforward application of Theorem~\ref{thm:yatracos}. Let us denote $\mathcal{D}_1$ the half of the sample on which we construct our density estimates and $\mathcal{D}_2$ the half on which we run the selection procedure. 
By Theorem~\ref{thm:yatracos}, setting $\delta = 1/n$, we obtain that conditioning on the first half of the sample, with probability at least $1 - 1/n$ we select $\psi$ such that,
\begin{align*}
\int_z \|\psi(x|z) - p(x|z)\|_1 p_Z(z) dz \lesssim  \min_{j \in [N]} \int_z \|\widehat{f}_{j} - p(x|z)\|_1 p_Z(z) dz+  \sqrt{ \frac{\log n}{n}}.
\end{align*}
Let us denote by $E$ the event on which this guarantee holds, and denote by $j^* \in [N]$ a density estimate constructed with (nearly) optimal choices of the tuning parameters.
The expected error (the expectation taken over all samples), can be bounded as:
\begin{align*}
\mathbb{E} \left[\int_z \|\psi(x|z) - p(x|z)\|_1 p_Z(z) dz\right] \leq \mathbb{E}_{\mathcal{D}_1, \mathcal{D}_2}  \left[\int_z \|\psi(x|z) - p(x|z)\|_1 p_Z(z) dz | E\right] + \frac{2}{n},
\end{align*}
by noting that the error is always atmost $2$ since both $\psi$ and $p$ are valid densities (and the $L_1$-loss is upper bounded by 2 for densities). Finally, we note that,
\begin{align*}
\mathbb{E} \left[\int_z \|\psi(x|z) - p(x|z)\|_1 p_Z(z) dz\right] &\lesssim \mathbb{E}_{\mathcal{D}_1}  \left[ \min_{j \in [N]} \int_z \|\widehat{f}_{j} - p(x|z)\|_1 p_Z(z) dz+  \sqrt{ \frac{\log n}{n}} \right] + \frac{2}{n}, \\
&\lesssim  \mathbb{E}_{\mathcal{D}_1}  \left[\int_z \|\widehat{f}_{j^*} - p(x|z)\|_1 p_Z(z) dz\right] + \sqrt{ \frac{\log n}{n}}, \\
&\lesssim n^{\frac{-1}{\beta^{-1} d_X + \gamma^{-1} d_Z + 2}} +  \sqrt{ \frac{\log n}{n}} \\
&\lesssim n^{\frac{-1}{\beta^{-1} d_X + \gamma^{-1} d_Z + 2}},
\end{align*}
where the last inequality follows by noting that for any finite $\beta$ or $\gamma$ the rate of conditional density estimation is strictly slower than $\mathcal{O}(\sqrt{(\log n)/n})$.

%

\section{Discussion}

In this paper we looked at the problem of conditional density estimation under a weighted absolute value loss function. We first demonstrated that if one imposes smoothness only on the conditional densities $p_{X|Z}(x|z)$ with respect to $x$, conditional density estimation is impossible in a minimax sense regardless of the marginal density $p_Z$ (which may even be known to the statistician). We then derived the minimax rate of estimation and showed an adaptive estimator which achieves the rate without knowledge of the smoothness parameters.

An interesting question that we intend to investigate in our future work is to generalize our results to an $L_p$ loss function:
\begin{align*}
\int |\hat p_{X|Z}(x|z) - p_{X|Z}(x|z)|^p p_Z(z) dz,
\end{align*} 
for some $p \geq 1$. We anticipate that such a modification will require a smoothness assumption stronger than TV smoothness. It will be interesting to see whether one can show that TV smoothness is not sufficient to analyze the $L_p$ loss function for $p > 1$. In addition we are interested in quantifying higher order TV smoothness and studying the problem of conditional density estimation for such densities.

\section{Acknowledgements}
The authors are grateful to Larry Wasserman for helpful discussions. SB was partially supported by NSF grants DMS-1713003 and CCF-1763734.

\bibliographystyle{unsrtnat}
\bibliography{michaelsbibfile}

\newpage
\appendix

\section{Additional Technical Results}
\subsection{Proof of Lemma \ref{lemma:p1p2}}
We begin by verifying the first and second claims. Notice that $p_1$ belongs to the H\"older class with any smoothness value $\beta$. For the $\beta \leq 1$ cases, $p_1$ is H\"older smooth with constant $\frac{2(1 - c)}{d_X^{\beta}}$; furthermore linear functions are H\"older smooth for any $\beta > 1$ and any $\holdercon > 0$.
Now we consider $p_2$. 
We note that $p_1 \leq 2$. It immediately follows that if we define $p_2 = 2 - p_1$ then $p_2$ is a valid density in the same H\"older class as $p_1$, and  $\frac{1}{2} p_1 + \frac{1}{2} p_2 = U([0,1]^{d_X})$.

Finally we examine the TV distance
\begin{align*}
\TV(U([0,1]^{d_X}), p_1) = \TV(U([0,1]^{d_X}), p_2) = \frac{1}{2} \int \left| 1 - \left( \alpha \sum_{i \in [d_x]} \frac{x_i}{d_X} + c \right) \right| dx.
\end{align*}
We will lower bound this from below:
\begin{align*}
\TV(U([0,1]^{d_X}), p_1) 
& \geq \frac{1}{2} \int 3 \left( \frac{\left| 1 - \left( \alpha \left(\sum_{i \in [d_x]} \frac{x_i}{d_X}\right) + c \right) \right|}{3}\right)^{2} dx\\
& =  \frac{1}{6} \int \left( 1 - 2 \alpha \sum_{i \in [d_x]} \frac{x_i}{d_X} - 2c + \alpha^2 \left(\sum_{i \in [d_x]} \frac{x_i}{d_X}\right)^2 + 2\alpha c \sum_{i \in [d_x]} \frac{x_i}{d_X} + c^2 \right) dx \\
& = \frac{1}{6} \left( (1 - c)^2 - (1 - c) \alpha + \alpha^2 (\frac{1}{3d_X} + \frac{d_X - 1}{4d_X}) \right),
\end{align*}
where the first inequality holds since $\frac{\left| 1 - \left( \alpha \left(\sum_{i \in [d_x]} \frac{x_i}{d_X}\right) + c \right) \right|}{3} \leq 1$, and we used $\int \sum_{i \in [d_X]} \frac{x_i}{d_X} dx = \frac{1}{2}$ and $\int (\sum_{i \in [d_X]} \frac{x_i}{d_X})^2 dx = \frac{\int x_1^2 d x_1}{d_X} + \frac{d_X - 1}{d_X} \int x_1 d x_1\int x_2 d x_2 = \frac{1}{3 d_X} + \frac{d_X - 1}{4d_X}$. Substituting in $\alpha = 2(1-c)$ from previous observations we get
\begin{align*}
\TV(U([0,1]^{d_X}), p_1) & \geq  \frac{1}{6} \left( (1 - c)^2 - 2(1 - c)^2 +  4(1 - c)^2 (\frac{1}{3d_X} + \frac{d_X - 1}{4d_X}) \right) \\
& = \frac{1}{6} \left( (1 - c)^2 \left( \frac{4}{3d_X} + \frac{d_X - 1}{d_X} - 1 \right) \right) \\
& = \frac{(1 - c)^2}{18d_X}\\
& \geq 0.
\end{align*}

\section{Proofs of Section 3} \label{app:upper}

\begin{proof}[Proof of Lemma \ref{kernel:construct}]

It is easy to see that the construction in \cite[Proposition 1.3]{tsybakov09introduction} provides kernels satisfying 
\begin{align*}
\int |K_i(u)|| |u|^{\kappa} du < \infty,
\end{align*}
for any fixed $\kappa \geq 0$. This is so since by construction $K_i(u)$ are Legendre polynomials supported on $[-1,1]$. In addition they also satisfy $\int K_i^2(u) du < \infty$. Furthermore since each kernel is of order $\ell$ we have that for any non-negative integer $k \leq \ell$: $\int K_i(u) u^{k} du = 0 $.  The statement of the lemma follows immediately by combining these three properties. 
%

\end{proof}

\begin{proof}[Proof of Lemma \ref{lemma:ptilde:density}]
Suppose that
\begin{align*}
\int |p_{X|Z}(x|z) - \hat{p}_{X|Z}(x|z)| dx \leq \epsilon_n(z),
\end{align*}
and $\hat{p}_{X|Z = z} \not \equiv 0$. Now consider the set $S = \Set{x}{\hat{p}_{X|Z}(x|z) \geq 0}$. Observe that
\begin{align*}
\int_S |p_{X|Z}(x|z) - \hat{p}_{X|Z}(x|z)| dx \leq \int |p_{X|Z}(x|z) - \hat{p}_{X|Z}(x|z)| dx \leq \epsilon_n(z).
\end{align*}
Since on the set $S^c$ we have $|p_{X|Z}(x|z) - \hat{p}_{X|Z}(x|z)| = p_{X|Z}(x|z) - \hat{p}_{X|Z}(x|z)$, and we know $p_{X|Z}(x|z) \geq 0$, we can conclude by the above inequality that
\begin{align*}
\int_{S^c} - \hat{p}_{X|Z}(x|z) dx \leq \epsilon_n(z).
\end{align*}
In addition, since $\int \hat{p}_{X|Z}(x|z) dx = 1$, we have
\begin{align*}
C = \int (\hat p_{X|Z}(x |z))_+ d x \geq 1 - \epsilon_n(z).
\end{align*}
By the triangle inequality we also have
\begin{align*}
C \leq \int |\hat p_{X|Z}(x |z)| d x \leq \int |\hat p_{X|Z}(x|z) - p_{X|Z}(x|z)| d x + 1 \leq 1 + \epsilon_n(z).
\end{align*}
Finally, we know the following holds
\begin{align*}
\int |p_{X|Z}(x|z) - (\hat{p}_{X|Z}(x|z))_+ | dx \leq \int |p_{X|Z}(x|z) - \hat{p}_{X|Z}(x|z)| dx \leq \epsilon_n(z).
\end{align*}
Combining the above results gives us
\begin{align*}
\int |p_{X|Z}(x|z) - C^{-1}(\hat{p}_{X|Z}(x|z))_+ | dx & \leq \int |p_{X|Z}(x|z) - (\hat{p}_{X|Z}(x|z))_+ | dx + \int \frac{|1 - C|}{C} (\hat{p}_{X|Z}(x|z))_+ | dx \\
& \leq \epsilon_n(z) + |1 - C| \\
& \leq 2 \epsilon_n(z).
\end{align*}
Next, notice that when $\hat p_{X|Z = z} \equiv 0$ then we have $\int |p_{X|Z}(x|z) - \hat{p}_{X|Z}(x|z)| dx = 1$, whereas, 
\begin{align*}
\int |p_{X|Z}(x|z) - \bar{p}_{X|Z}(x|z)| dx \leq 2,
\end{align*}
hence the same bound as above applies. Finally integrating the bound over $z$ completes the proof.
\end{proof}

\subsection{Lemmas of Section \ref{proof:of:upper:bound}}\label{lemmas:sec3:appendix}
\begin{proof}[Proof of Lemma \ref{lagrange:multipliers:lemma}]
Let 
\begin{align*}
\sum_{\bar{j}} p_{\bar{j}} (1 - p_{\bar{j}})^n
\end{align*}
 be the objective function that we try to maximize. We will use the Karush-Kuhn-Tucker (KKT) conditions and we will subject the objective to the constraints $p_{\bar{j}} \geq 0$ for all $\bar{j}$ and $\sum_{\bar{j}} p_{\bar{j}} = 1$. We introduce the KKT multipliers $\lambda_{\bar{j}} \leq 0$ for all $\bar{j}$, and $\mu$ corresponding to the constraints respectively. Then taking the derivative with respect to some $\bar{j}$ we have the conditions
\begin{align*}
(1-p_{\bar{j}})^n - n p_{\bar{j}} (1 - p_{\bar{j}})^{n-1} - \lambda_{\bar{j}} p_{\bar{j}} + \mu = 0
\end{align*}
and by complementary slackness $\lambda_{\bar{j}} p_{\bar{j}} = 0$ for all $\bar{j}$. 

Let $S = \Set{\bar{j}}{ p_{\bar{j}} \neq 0}$, which means from the conditions that on the set $S$, all $\lambda_{\bar{j}} = 0$ and $(1-p_{\bar{j}})^n - n p_{\bar{j}} (1 - p_{\bar{j}})^{n-1} = - \mu$.  We can write this as $f(x) = (1-(n+1)x)(1-x)^{n-1}$ where $x = p_{\bar{j}}$. Clearly $f$ is decreasing on $[0,\frac{1}{n+1}]$ and $f(\frac{1}{n+1}) = 0$. Since $|S| \leq m^{d_Z} \ll  n+1$ (by our construction of $m$) and $\sum_{\bar{j}} p_{\bar{j}} = 1$, there exists $\bar{k} \in S$ such that $p_{\bar{k}} \geq \frac{1}{|S|} \geq \frac{1}{m^{d_Z}} \gg \frac{1}{n+1}$. But $(1-(n+1)p_{\bar{k}})(1-p_{\bar{k}})^{n-1} < 0$, so it follows that $\mu > 0$. Now observe $f'(x) =  n(1-x)^{n-2}((n+1)x - 2)$. This shows that on the interval $[\frac{1}{n+1},1]$, $f$ changes from decreasing to increasing exactly once, at the point $\frac{2}{n+1} > \frac{1}{n+1}$. This implies that the equations $f(x) = - \mu$ for some $\mu > 0$ and $x \in [\frac{1}{n+1}, 1]$ can have at most two solutions. 

Then simply divide $S = S_1 \cup S_2$, where all $p_{\bar{j}}$ on $S_1$ are equal to some $v$, and all $p_{\bar{j}}$ on $S_2$ are equal to some $w$, such that $|S_1|v + |S_2|w = 1, v,w \in [\frac{1}{n+1},1]$ and $f(v) = f(w) > 0$. Substituting this in, our objective function becomes
\begin{align*}
\sum_{\bar{j}} p_{\bar{j}} (1 - p_{\bar{j}})^n = |S_1| v(1-v)^{n} + |S_2| w(1-w)^{n}.
\end{align*}
But the function $(1-x)^n$ is convex, so by Jensen's inequality it follows that
\begin{align*}
|S_1| v(1-v)^{n} + |S_2| w(1-w)^{n} \leq (1 - |S_1|v^2 - |S_2|w^2)^n,
\end{align*}
which is maximized when we minimize $|S_1|v^2 + |S_2|w^2$ under the constraint $|S_1|v + |S_2|w = 1$. This is done by rearranging to get $v = \frac{1 - |S_2|w}{|S_1|}$, so the minimum is achieved at $v = w = \frac{1}{|S_1| + |S_2|}$. This result shows that $p_{\bar{j}}$ must have the same value of $\frac{1}{|S|}$ for all $\bar{j}$, which is what one would intuitively expect. Our objective function then becomes
\begin{align*}
\sum_{\bar{j}} p_{\bar{j}} (1 - p_{\bar{j}})^n = |S| \frac{1}{|S|} \bigg(1 - \frac{1}{|S|}\bigg)^n = \bigg(1 - \frac{1}{|S|}\bigg)^n.
\end{align*}
In order to maximize this, we want $|S|$ to be as large as possible, which in this case is $m^{d_Z}$. This completes the proof.
\end{proof}

\begin{lemma}\label{lemma:expected}
Under the same assumptions as Theorem \ref{upper:bound:holder:thm}, the estimator
\begin{align*}
\hat{p}_{X,\bar{j}}(x) = \frac{\sum_{i \in [n]} \mathbbm{1}(Z_i \in A_{\bar{j}}) K(\frac{X_i - x}{h})}{h^{d_X} \sum_{i \in [n]} \mathbbm{1} (Z_i \in A_{\bar{j}})}
\end{align*}
for some $\bar{j} \in [m]^{d_Z}$ has the expected value
\begin{align*}
\mathbb{E}[\hat{p}_{X,\bar{j}}(x)] = h^{-d_X} \mathbb{E} \bigg[ K\left( \frac{X-x}{h} \right) \bigg| Z \in A_{\bar{j}} \bigg] (1 - \mathbb{P}(Z \in A_{\bar{j}}^c)^n).
\end{align*}
\end{lemma}

\begin{proof}[Proof of Lemma \ref{lemma:expected}]
Using the law of total expectation we have
\begin{align*}
\mathbb{E}[\hat{p}_{X,\bar{j}}(x)]  = \sum_{S \in 2^{[n]}} \mathbb{E} [\hat{p}_{X,\bar{j}}(x) | Z_i \in A_{\bar{j}}, i \in S, Z_i \in A_{\bar{j}}^c, i \in S^c] \mathbb{P}(Z \in A_{\bar{j}})^{|S|} \mathbb{P}(Z \in A_{\bar{j}}^c)^{n - |S|},
\end{align*}
where the condition in the conditional expectation means that 
\begin{align*}
\sum_{i \in [n]} \mathbbm{1}(Z_i \in A_{\bar{j}})K\bigg(\frac{X_i - x}{h}\bigg) = \sum_{i \in S} K\bigg(\frac{X_i - x}{h}\bigg)
\end{align*}
and that 
\begin{align*}
\sum_{i \in [n]} \mathbbm{1}(Z_i \in A_{\bar{j}}) = |S|.
\end{align*}
Then the conditional expectation can be rewritten as:
\begin{align*} 
& \mathbb{E} [\hat{p}_{X,\bar{j}}(x) | Z_i \in A_{\bar{j}}, i \in S, Z_i \in A_{\bar{j}}^c, i \in S^c] \\
= & \: \mathbb{E} \bigg[\frac{\sum_{i \in [n]} \mathbbm{1}(Z_i \in A_{\bar{j}})K\big(\frac{X_i - x}{h}\big)}{h^{d_X} \sum_{i \in [n]} \mathbbm{1}(Z_i \in A_{\bar{j}})} \bigg| Z_i \in A_{\bar{j}}, i \in S, Z_i \in A_{\bar{j}}^c, i \in S^c\bigg] \\
= & \: \mathbb{E} \bigg[ \frac{\sum_{i \in S} K\big(\frac{X_i - x}{h}\big)}{h^{d_X} |S|} \bigg|  Z_i \in A_{\bar{j}}, i \in S\bigg] \\
= & \: \frac{1}{h^{d_X}|S|} \sum_{i \in S} \mathbb{E} \bigg[ K\big(\frac{X_i - x}{h}\big) \bigg| Z_i \in A_{\bar{j}}\bigg]\\
= & \: h^{-d_X} \mathbb{E}\bigg[K\bigg(\frac{X-x}{h}\bigg) \bigg| Z \in A_{\bar{j}}\bigg].
\end{align*}

Now notice that 
\begin{align*}
\sum_{S \in 2^{[n]}} \mathbb{P}(Z \in A_{\bar{j}})^{|S|} \mathbb{P}(Z \in A_{\bar{j}}^c)^{n - |S|} = 1
\end{align*}
and there is one special case when $S$ is the empty set $\emptyset$, where by definition the estimator $\hat{p}_{X,\bar{j}}(x) = \frac{0}{0} := 0$. This occurs when $|S| = 0$ with a corresponding probability of $\mathbb{P}(Z \in A_{\bar{j}}^c)^n$. So we subtract this probability, giving:
\begin{align*}
\mathbb{E}[\hat{p}_{X,\bar{j}}(x)] = h^{-d_X} \mathbb{E} \bigg[K\left( \frac{X-x}{h} \right) | Z \in A_{\bar{j}} \bigg] (1 - \mathbb{P}(Z \in A_{\bar{j}}^c)^n)
\end{align*}
as desired.
\end{proof}

\begin{lemma}\label{lemma:upper}
Under the same assumptions as Theorem \ref{upper:bound:holder:thm}, for some $\bar{j} \in [m]^{d_Z}$ we have that:
\begin{align*}
\bigg| h^{-d_X} \mathbb{E} \bigg[ K\left( \frac{X-x}{h} \right) | Z \in A_{\bar{j}} \bigg] - p_{X|Z}(x|z \in A_{\bar{j}}) \bigg| \leq C h^\beta
\end{align*}
for some constant $C$. 
\end{lemma}

\begin{proof}[Proof of Lemma \ref{lemma:upper}]
By assumption that $K$ is a kernel of order $\ell$, it follows that:
\begin{align*}
& \bigg| h^{-d_X} \int K\left( \frac{y-x}{h} \right) p_{X|Z}(y|z \in A_{\bar{j}})dy - p_{X|Z}(x|z \in A_{\bar{j}}) \bigg| \\
= &  \bigg| \int K (u)p_{X|Z}(x+uh|z \in A_{\bar{j}})du - p_{X|Z}(x|z \in A_{\bar{j}}) \bigg| \\
= &  \bigg| \int K (u) \big[p_{X|Z}(x+uh|z \in A_{\bar{j}}) - p_{X|Z}(x|z \in A_{\bar{j}}) \big] du \bigg|.
\end{align*}
%
%
%
But by Lemma \ref{lemma:holder:ext} $p_{X|Z}(x+uh|z \in A_{\bar{j}})$ is $\ell$ times differentiable. Then the Taylor series is:
%
%
%
\begin{align*}
p_{X|Z}(x+uh|z \in A_{\bar{j}}) = \sum_{\|\alpha\|_1 < \ell} \frac{D^{\alpha} p_{X|Z}(x|z \in A_{\bar{j}})}{\alpha!} (uh)^{\alpha} +  \sum_{\|\alpha\|_1 = \ell} \frac{D^{\alpha} p_{X|Z}(x + \tau u h|z \in A_{\bar{j}})}{\alpha!} u^{\alpha} h^{\ell}, 
\end{align*}
where $\|\alpha\|_1 = \ell$ and $\tau \in [0,1]$. Substituting this back in cancels out the first summation and $p_{X|Z}(x|z \in A_{\bar{j}})$ (since the kernel is of order $\ell$), giving: 
%
%
%
\begin{align*}
& \bigg| \int K (u) [p_{X|Z}(x+uh|z \in A_{\bar{j}}) - p_{X|Z}(x|z \in A_{\bar{j}})] du \bigg| \\
= & \bigg| \int K(u)  \sum_{\|\alpha\|_1 = \ell} \frac{D^{\alpha} p_{X|Z}(x + \tau u h|z \in A_{\bar{j}})}{\alpha!} u^{\alpha} h^{\ell} du \bigg| \\
 =& \bigg| \int K(u)  \sum_{\|\alpha\|_1 = \ell} \frac{D^{\alpha} p_{X|Z}(x + \tau u h|z \in A_{\bar{j}})}{\alpha!} u^{\alpha} h^{\ell} du - \int K(u)  \sum_{\|\alpha\|_1 = \ell} \frac{D^{\alpha} p_{X|Z}(x |z \in A_{\bar{j}})}{\alpha!} u^{\alpha} h^{\ell} du \bigg| \\
 \leq & \int |K(u)| \frac{|u^{\alpha} h^{\ell}|}{\alpha!} \sum_{\|\alpha\|_1 = \ell} \big|D^{\alpha} p_{X|Z}(x + \tau u h|z \in A_{\bar{j}}) - D^{\alpha} p_{X|Z}(x |z \in A_{\bar{j}}) \big|  du \\
 \leq  & \int |K(u)| \frac{|u^{\alpha} h^{\ell}|}{\alpha!} \sum_{\|\alpha\|_1 = \ell} \holdercon \|\tau u h\|_1^{\beta - \ell} du \\
= & h^{\beta}   |\tau|^{\beta - \ell} \frac{\holdercon }{\alpha!} \sum_{\|\alpha\|_1 = \ell} \int |K(u)| \cdot |u^{\alpha}| \cdot \|u\|_1^{\beta - \ell} du.
\end{align*}
In the above, we used the fact that $p_{X|Z}(x|z \in A_{\bar{j}})$ also follows the H\"older smoothness condition because of Lemma \ref{lemma:holder:ext}.

Since in $\mathbb{R}^d$ all norms are equivalent, we know that for any $q \geq 1$, there exists $C_q$ such that $\left( \sum_i |y_i|^{q}\right)^{\frac{1}{q}} \leq C_q \sum_i |y_i| = \|y\|_1$. Now let $y_i = |u_i|^q$ where $q = (\beta - \ell)^{-1}$. It follows that
\begin{align*}
\|u\|_1^{\beta - \ell} = \left( \sum_i |u_i| \right)^{\beta - \ell} \leq C_{(\beta - \ell)^{-1}} \sum_i |u_i|^{\beta - \ell}
\end{align*}
for some constant $C_{\beta - \ell}$. Then

\begin{align*}
 \sum_{\|\alpha\|_1 = \ell}|u^{\alpha}| \cdot \|u\|_1^{\beta - \ell} \leq C_{(\beta - \ell)^{-1}}  \sum_{\|\alpha\|_1 = \ell} \sum_i  |u^{\alpha}| |u_i|^{\beta - \ell} \lesssim   \sum_{\|\alpha\|_1 = \beta} |u^{\alpha}|,
\end{align*}
where in the last inequality the constant may depend on $\beta- \ell$, and the dimension $d_X$. Since we are assuming $\int |K(u)| |u^{\alpha}| du < \infty$ for all $\|\alpha\|_1 \leq \beta, \alpha \in \RR_+^{d_X}$, this means that 

\begin{align*}
& \bigg| \int K (u) [p_{X|Z}(x+uh|z \in A_{\bar{j}}) - p_{X|Z}(x|z \in A_{\bar{j}})] du \bigg| \leq  C h^{\beta}   |\tau|^{\beta - \ell} \leq  C h^{\beta}.
\end{align*}

\end{proof}

\begin{lemma}\label{lemma:variance}
Under the same assumptions as Theorem \ref{upper:bound:holder:thm}, the estimator 
\begin{align*}
\hat{p}_{X,\bar{j}}(x) = \frac{\sum_{i \in [n]} \mathbbm{1}(Z_i \in A_{\bar{j}}) K(\frac{X_i - x}{h})}{h^{d_X} \sum_{i \in [n]} \mathbbm{1} (Z_i \in A_{\bar{j}})}
\end{align*}
for some $\bar{j} \in [m]^{d_Z}$ has its variance upper bounded as
\begin{align*}
\var[\hat{p}_{X,\bar{j}}(x)] \leq \frac{C}{nh^{d_X} \mathbb{P}(Z \in A_{\bar{j}})} + K \mathbb{P}(Z \in A_{\bar{j}}^c)^n.
\end{align*}
where $C, K$ are constants.
\end{lemma}

\begin{proof}[Proof of Lemma \ref{lemma:variance}]
By the law of total variance we have:
\begin{align}\label{total:var:eq}
\var(\hat{p}_{X,\bar{j}}(x)) = & \mathbb{E}[\var(\hat{p}_{X,\bar{j}}(x)|Z_i \in A_{\bar{j}}, i \in S, Z_i \in A_{\bar{j}}^c, i \in S^c)] \nonumber \\
& \quad + \var_S[\mathbb{E}[\hat{p}_{X,\bar{j}}(x) | Z_i \in A_{\bar{j}}, i \in S, Z_i \in A_{\bar{j}}^c, i \in S^c]].
\end{align}
We proceed to bound the two terms separately.
\\
\\
\textbf{Bounding the first term}\\\\
The conditions in the expectation means that
\begin{align*}
\sum_{i \in [n]} \mathbbm{1}(Z_i \in A_{\bar{j}})K\bigg(\frac{X_i - x}{h}\bigg) = \sum_{i \in S} K\bigg(\frac{X_i - x}{h}\bigg), \quad \sum_{i \in [n]} \mathbbm{1}(Z_i \in A_{\bar{j}}) = |S|.
\end{align*}
We first consider the variance term inside the expectation:
\begin{align*}
& \var\big(\hat{p}_{X,\bar{j}}(x)\big |Z_i \in A_{\bar{j}}, i \in S, Z_i \in A_{\bar{j}}^c, i \in S^c\big) \\
= & \var\bigg(\frac{\sum_{i \in S} K(\frac{X_i - x}{h})}{h^{d_X} |S|} \bigg| Z_i \in A_{\bar{j}}, i \in S\bigg) \\
= & \frac{1}{h^{2d_X} |S|^2} \sum_{i \in S} \var\bigg(K\bigg(\frac{X_i - x}{h}\bigg) \bigg| Z_i \in A_{\bar{j}}\bigg) \\
\leq & \frac{1}{h^{2d_X} |S|^2} \sum_{i \in S}  \mathbb{E}\bigg[K^2\bigg(\frac{X_i - x}{h}\bigg) \bigg| Z_i \in A_{\bar{j}}\bigg].
\end{align*}
%
Then when $S \neq \emptyset$ we have
\begin{align*}
& \var(\hat{p}_{X,\bar{j}}(x)|Z_i \in A_{\bar{j}}, i \in S, Z_i \in A_{\bar{j}}^c, i \in S^c) \\
\leq & \frac{1}{h^{2d_X} |S|} \mathbb{E}\bigg[K^2\bigg(\frac{X - x}{h}\bigg) \bigg| Z \in A_{\bar{j}}\bigg] \\
\leq & \frac{2}{h^{2d_X} (|S| + 1)} \mathbb{E}\bigg[K^2\bigg(\frac{X - x}{h}\bigg) \bigg| Z \in A_{\bar{j}}\bigg],
\end{align*}
where the last step makes a small sacrifice in the tightness of the bound in order to allow a very useful identity to be applied later. As for the other case when $S = \emptyset$, by definition the estimator $\hat{p}_{X,\bar{j}}(x) = 0$, so the above still holds.



We proceed to bound the expectation $\mathbb{E}\big[K^2\big(\frac{X - x}{h}\big) \big| Z \in A_{\bar{j}}\big]$. 
First consider $K^*$, a bounded kernel of order $\ell$, not necessarily equal to $K$. Now
notice that  $p_{X|Z}(x|z \in A_{\bar{j}}) \leq p_{\max} < \infty$. This can be proven by applying lemma \ref{lemma:upper} and setting $h=1$ to get
\begin{align*}
\bigg|\int K^* \left( y-x \right) p_{X|Z}(y|z \in A_{\bar{j}})dy - p_{X|Z}(x|z \in A_{\bar{j}}) \bigg| \leq C.
\end{align*}
It follows that
\begin{align*}
p_{X|Z}(x|z \in A_{\bar{j}}) \leq C + \int |K^* \left( y-x \right)| p_{X|Z}(y|z \in A_{\bar{j}})dy \leq C + K^*_{\max} < \infty,
\end{align*}
where $K^*_{\max} = \sup_{u \in \mathbb{R}^{d_X}}|K^*(u)|$.


Now we have
\begin{align}\label{lemma:variance:expected:eq}
\mathbb{E}\bigg[K^2\bigg(\frac{X - x}{h}\bigg) \bigg| Z \in A_{\bar{j}}\bigg] = \int K^2 \left( \frac{y - x}{h} \right) p_{X|Z}(y | z \in A_{\bar{j}}) dy \leq p_{\max} h^{d_X} \int K^2(u) du \leq C h^{d_X},
\end{align}
where $C \geq p_{\max} \int K^2(u) du$ is a constant, since both the conditional density and $\int K^2(u) du$ are upper bounded.

Substituting this back into the first term of \eqref{total:var:eq} we get
\begin{align*}
&\mathbb{E}[\var(\hat{p}_{X,\bar{j}}(x)|Z_i \in A_{\bar{j}}, i \in S, Z_i \in A_{\bar{j}}^c, i \in S^c)] \\
\leq & \mathbb{E} \bigg[ \frac{2}{h^{2d_X} (|S| + 1)} C h^{d_X} \bigg] \\
= &  \frac{2C}{h^{d_X}}  \mathbb{E} \bigg[ \frac{1}{|S| + 1} \bigg].
\end{align*}

By Lemma \ref{lemma:identity} we have
\begin{align*}
\mathbb{E} \bigg[\frac{1}{|S|+1} \bigg] = \frac{1 - \mathbb{P}(Z \in A_{\bar{j}}^c)^{n+1}}{(n+1)\mathbb{P}(Z \in A_{\bar{j}})} \leq \frac{1}{(n+1)\mathbb{P}(Z \in A_{\bar{j}})}.
\end{align*}
Then the first term is bounded by 
\begin{align*}
\mathbb{E}[\var(\hat{p}_{X,\bar{j}}(x)|Z_i \in A_{\bar{j}}, i \in S, Z_i \in A_{\bar{j}}^c, i \in S^c)] \leq \frac{2C}{h^{d_X}(n+1) \mathbb{P}(Z \in A_{\bar{j}})} \asymp \frac{C}{nh^{d_X} \mathbb{P}(Z \in A_{\bar{j}})}
\end{align*}
for some constant $C$.
\\
\\
\textbf{Bounding the second term}\\\\
Now we bound the second term from equation \eqref{total:var:eq}, and start by examining the inner expectation. When $S \neq \varnothing$
%
%
\begin{align*}
\mathbb{E} [\hat{p}_{X,\bar{j}}(x) | Z_i \in A_{\bar{j}}, i \in S, Z_i \in A_{\bar{j}}^c, i \in S^c] = h^{-d_X} \mathbb{E}\bigg[K\bigg(\frac{X-x}{h}\bigg) \bigg| Z \in A_{\bar{j}}\bigg]
\end{align*}
and otherwise when $S = \varnothing$
\begin{align*}
\mathbb{E} [\hat{p}_{X,\bar{j}}(x) | Z_i \in A_{\bar{j}}, i \in S, Z_i \in A_{\bar{j}}^c, i \in S^c] = 0.
\end{align*}
Then
\begin{align*}
& \var_S[\mathbb{E}[\hat{p}_{X,\bar{j}}(x) | Z_i \in A_{\bar{j}}, i \in S, Z_i \in A_{\bar{j}}^c, i \in S^c]] \\
= & \left( h^{-d_X} \mathbb{E}\bigg[K\bigg(\frac{X-x}{h}\bigg) \bigg| Z \in A_{\bar{j}}\bigg] \right) ^2 \mathbb{P}(Z \in A_{\bar{j}}^c)^n (1 - \mathbb{P}(Z \in A_{\bar{j}}^c)^n).
\end{align*}

Since $\int K^2(u) du < \infty$ is upper bounded, it follows that $\int |K(u)| du \leq \sqrt{\int K^2(u) du} < \infty$. Therefore by the same logic as how we arrived at the bound in equation \eqref{lemma:variance:expected:eq}, we have $\mathbb{E}[K(\frac{X-x}{h}) | Z \in A_{\bar{j}}] \leq C' h^{d_X}$ for some constant $C'$. Then we can simply bound the whole expression as:
\begin{align*}
& \var_S[\mathbb{E}[\hat{p}_{X,\bar{j}}(x) | Z_i \in A_{\bar{j}}, i \in S, Z_i \in A_{\bar{j}}^c, i \in S^c]] \\
\leq & {C'}^2  \mathbb{P}(Z \in A_{\bar{j}}^c)^n (1 - \mathbb{P}(Z \in A_{\bar{j}}^c)^n) \\
\leq & K \mathbb{P}(Z \in A_{\bar{j}}^c)^n
\end{align*}
for some constant $K$.\\

\noindent \textbf{Combining the terms}\\\\
For the variance \eqref{total:var:eq} we have split it into two terms and upper bounded them individually. It follows that
\begin{align*}
\var[\hat{p}_{X,\bar{j}}(x)] \leq \frac{C}{nh^{d_X} \mathbb{P}(Z \in A_{\bar{j}})} + K \mathbb{P}(Z \in A_{\bar{j}}^c)^n
\end{align*}
for some constants $C, K$ as desired.
\end{proof}

\begin{lemma}\label{lemma:identity}
We have the identity
\begin{align*}
\mathbb{E}\bigg[\frac{1}{|S|+1}\bigg] = \frac{1 - \mathbb{P}(Z \in A_{\bar{j}}^c)^{n+1}}{(n+1)\mathbb{P}(Z \in A_{\bar{j}})}.
\end{align*}
\end{lemma}

\begin{proof}
By definition of $|S|$, it can be regarded as a binomial distribution $|S| \sim Bin(n,p)$ where $p = \mathbb{P}(Z \in A_{\bar{j}})$. We also set $q = 1 - p = \mathbb{P}(Z \in A_{\bar{j}}^c)$. Then
\begingroup
\allowdisplaybreaks
\begin{align*}
\mathbb{E}\bigg[\frac{1}{|S|+1}\bigg] & = \sum_{k=0}^n \frac{1}{k+1} \mathbb{P}(|S| = k) \\
& = \sum_{k=0}^n \frac{1}{k+1} \frac{n!}{k! (n-k)!} p^k q^{n-k} \\
& = \frac{1}{(n+1)p} \sum_{k=0}^n {n+1 \choose k+1} p^{k+1}q^{n-k} \\
& = \frac{1}{(n+1)p} [ ( \sum_{t=0}^{n+1} {n+1 \choose t} p^t q^{(n+1) - t})  - q^{n+1} ] \\
& = \frac{1}{(n+1)p} [ ( p+q)^{n+1}  - q^{n+1} ] \quad \text{by the binomial theorem} \\
& = \frac{1}{(n+1)p} [ 1 - q^{n+1}] \\
& = \frac{1 - \mathbb{P}(Z \in A_{\bar{j}}^c)^{n+1}}{(n+1)\mathbb{P}(Z \in A_{\bar{j}})}.
\end{align*}
\endgroup
\end{proof}

\begin{lemma}\label{lemma:holder:ext}
Given the H\"older smoothness condition in Definition \ref{holder:smooth:def} with some smoothness $\beta$, the conditional density $p_{X|Z}(x|z \in A_{\bar{j}})$ also satisfies the same property. That is, it is $\ell = \floor{\beta}$ times differentiable and satisfies
\begin{align*}
\sup_{\alpha} | D^{\alpha} p_{X|Z}(x|z\in A_{\bar{j}}) - D^{\alpha} p_{X|Z}(x'|z\in A_{\bar{j}})| \leq \holdercon \|x - x'\|_1^{\beta - \ell}
\end{align*}
for all $\alpha$ such that $ \|\alpha\|_1 = \ell$, $\alpha \in \mathbb{N}_{0}^{d_X}$, where $\alpha = (\alpha_1,...,\alpha_{d_X})$
\end{lemma}

\begin{proof}
We first show that $p_{X|Z}(x|z \in A_{\bar{k}})$ is $\ell$ times differentiable. The Leibniz integral rule in higher dimensions allows switching the order of derivative and integration as follows
\begin{align*}
\frac{\partial}{\partial x_i} \left( \int_{a}^{b} f(x, z) dz \right) = \int_{a}^{b} \frac{\partial}{\partial x_i} f(x, z) dz,
\end{align*}
where all elements of $a, b$ are bounded. In context, let $a, b$ be the lower and upper bound vector of $A_{\bar{j}}$ and let $f(x, z) = p_{X|Z}(x|z) \frac{p_{Z}(z)}{\mathbb{P}(Z \in A_{\bar{j}})}$. Then by the Leibniz integral rule, we have
\begin{align*}
D^{\alpha}p_{X|Z}(x|z \in A_{\bar{j}}) = \int_{A_{\bar{j}}} D^{\alpha} p_{X|Z}(x|z) \frac{p_{Z}(z)}{\mathbb{P}(Z \in A_{\bar{j}})} dz, \quad \|\alpha\|_1 = i \text{,  for} \: 1 \leq i \leq \ell,
\end{align*}
thus proving that $p_{X|Z}(x|z \in A_{\bar{j}})$ is $\ell$ times differentiable.

Now, for an arbitrary $\alpha$ such that $\|\alpha\|_1 = \ell$, applying the Leibniz integral rule gives:
\begin{align*}
& | D^{\alpha} p_{X|Z}(x|z\in A_{\bar{j}}) - D^{\alpha} p_{X|Z}(x'|z\in A_{\bar{j}})| \\
\leq \: &  \int_{A_{\bar{j}}}\bigg| D^{\alpha} p_{X|Z}(x|z) - D^{\alpha} p_{X|Z}(x'|z) \bigg| \frac{p_{Z}(z)}{\mathbb{P}(Z \in A_{\bar{j}})} dz \\
\leq \: & \int_{A_{\bar{j}}} \holdercon \|x - x'\|_1^{\beta - \ell} \frac{p_{Z}(z)}{\mathbb{P}(Z \in A_{\bar{j}})} dz \\
= \: &  \holdercon\|x - x'\|_1^{\beta - \ell}.
\end{align*}
So $p_{X|Z}(x|z\in A_{\bar{j}})$ indeed follows the H\"older smoothness condition.
\end{proof}

%
%
%
%
%
%
%

\section{Proofs of Section 5} \label{app:example}

\begin{proof}[Proof of Theorem \ref{example:greater:thm}]
We first show that $p_{X|Z}(x|z) = \frac{g(x,z)}{\int g(x,z) d x}$ is H\"older smooth in $x$ for any fixed $z$. Notice that the denominator $\int g(x,z) d x \geq a \cdot \mu([0,1]^{d_X}) = a > 0$ is lower bounded by some constant. Then to show that $p_{X|Z}(x|z)$ is H\"older smooth it suffices to show that $g(x,z)$ is H\"older smooth in $x$. But we already required this in the assumptions, where taking all partial derivatives with respect to $x$ satisfies the H\"older condition $\sup_{\alpha}  | D^{\alpha} g(x,z) - D^{\alpha} g(x',z)| \leq C \| x - x'\|_1^{\beta - \ell}$.

Now we show that $p_{X|Z}(x|z)$ is Lipschitz smooth in $z$ by showing that its derivative is bounded. Without loss of generality take the partial derivative with respect to some $z_i \in Z$, then we have:
\begin{align*}
& \sup_{x,z} \bigg| \frac{\partial}{\partial z_i} p_{X|Z}(x|z)\bigg| =  \sup_{x,z} \left| \frac{\frac{\partial}{\partial z_i} g(x, z) \cdot \int g(x, z) dx - \int  \frac{\partial}{\partial z_i}  g(x, z) dx \cdot  g(x, z)}{(\int  g(x, z) dx)^2} \right|,
\end{align*}
where we used the Leibniz integral rule to change the order of the derivative and the integral. But $\int g(x, z) dx \geq a$, thus the denominator $(\int  g(x, z) dx)^2 \geq a^2$ is lower bounded by some positive constant. Then we just need to show that the numerator is bounded. We have:
\begin{align*}
 & \sup_{{x,z}} \left| \frac{\partial}{\partial z_i} g(x, z) \cdot \int g(x, z) dx - \int  \frac{\partial}{\partial z_i}  g(x, z) dx \cdot  g(x, z) \right| \\
 \leq \: & \sup_{{x,z}} \left| \frac{\partial}{\partial z_i} g(x, z) \cdot \int g(x, z) dx \right| +  \sup_{{x,z}} \left| \int  \frac{\partial}{\partial z_i}  g(x, z) dx \cdot  g(x, z) \right|.
\end{align*}
Recall that $g(x, z)$ is $\ell$ times differentiable, where $\ell = \floor{\beta} \geq 1$. It follows that $g(x, z)$ and its derivatives up to order $\ell$ are continuous. Furthermore, since it is defined on a compact space $[0,1]^{d}$, $g(x, z)$ and $\frac{\partial}{\partial z_i} g(x, z)$ are bounded. Then each term in the expression above are bounded, and thus the numerator is bounded. So we have shown that the first derivative $ \sup_{{x,z}} | \frac{\partial}{\partial z_i} p_{X|Z}(x|z)| \leq K$ is bounded by some constant, and therefore $p_{X|Z}$ is Lipschitz smooth in $z$.  Substituting this result into the TV smoothness condition (see Definition \ref{TV:smooth:def}) we have:
\begin{align*}
\|p_{X|Z = z} - p_{X|Z = z'} \|_1 & = \int |p_{X|Z }(x | z) - p_{X|Z }(x | z')| dx \\
& \leq \int_0^1 K \| z - z'\|_1 dx \\
& = K \| z - z'\|_1
\end{align*}
as desired. So $p_{X|Z}$ is indeed H\"older smooth and TV smooth (hence it is also $\gamma$-TV smooth).
\end{proof}

\begin{proof}[Proof of Theorem \ref{example:less:thm}] 
We will first show that the function $p_{X|Z}(x | z)$ is H\"older smooth. To see this note that $\ell = \floor{\beta} = 0$ so we do not take partial derivatives, and
\begin{align*}
|p_{X|Z}(x|z) - p_{X|Z}(x'|z)| =  \frac{|\exp(g(x,z)) - \exp(g(x',z))|}{\int \exp(g(x,z)) d x} \leq \exp(M) |\exp(g(x,z)) - \exp(g(x',z))|.
\end{align*} 
Next let $g = g(x,z)$ and $g' = g(x',z)$ for brevity. We have
\begin{align*}
|e^g - e^{g'}| \leq  |g - g'|\sum_{k = 1}^{\infty} \frac{\sum_{i = 0}^{k-1} |g|^i |g'|^{(k-1-i)}}{k!} \leq  |g - g'| \exp(M) \leq C\exp(M) \|x - x'\|_1^\beta,
\end{align*}
and we conclude that
\begin{align*}
|p_{X|Z}(x|z) - p_{X|Z}(x'|z)| \leq C \exp(2M)\|x - x'\|_1^\beta. 
\end{align*}
Next we will control the quantity $\|p_{X|Z = z} - p_{X|Z = z'}\|_1$. To see this we note that
\begin{align*}
\|p_{X|Z = z} - p_{X|Z = z'}\|_1 & = \int |p_{X|Z}(x|z) - p_{X|Z}(x|z')| dx \\
& = \int \bigg(\frac{\max(p_{X|Z}(x|z), p_{X|Z}(x|z'))}{\min(p_{X|Z}(x|z), p_{X|Z}(x|z'))} - 1\bigg)\min(p_{X|Z}(x|z), p_{X|Z}(x|z'))  dx\\
& \leq  \int \bigg(\frac{\max(p_{X|Z}(x|z), p_{X|Z}(x|z'))}{\min(p_{X|Z}(x|z), p_{X|Z}(x|z'))} - 1\bigg)p_{X|Z}(x|z)dx.
\end{align*}
Suppose now that the function $\log p_{X|Z}(x | z)$ is H\"older with constants $K$ and $\gamma$ in $z$ then the above can be bounded as 
\begin{align*}
\|p_{X|Z = z} - p_{X|Z = z'}\|_1 & \leq \int (\exp(K \|z - z'\|^\gamma_1) - 1)p_{X|Z}(x|z)dx \\
& = K\|z-z'\|^\gamma_1 + \sum_{k \geq 2}(K\|z-z'\|^\gamma_1)^k/k! \\
& \leq K\|z-z'\|^\gamma_1 + K\|z-z'\|^\gamma_1\sum_{k \geq 2} (Kd_Z)^{k-1}/k! \\
& = K\|z-z'\|^\gamma_1 + K\|z-z'\|^\gamma_1 (e^{K d_Z} - 1 - Kd_Z)/(Kd_Z) \\
& = B \|z-z'\|^\gamma_1,
\end{align*}
where $B = K(1 + (e^{Kd_Z} - 1 - K d_Z)/(Kd_Z))$. It remains to show that $\log p_{X|Z}(x | z)$ is H\"older with constants $K$ and $\gamma$. Consider the difference
\begin{align*}
\log p_{X | Z}(x | z) - \log p_{X | Z}(x | z') & =  g(x,z) - g(x, z') - \log \frac{\int \exp(g(x, z)) dx}{\int \exp(g(x, z')) dx} \\
&\leq C\|z - z'\|^\gamma_1 - \log \frac{\int \exp(g(x, z) - g(x, z'))  \exp(g(x,z'))dx}{\int \exp(g(x, z')) dx}\\
& \leq C\|z - z'\|^\gamma_1  + \frac{\int (g(x, z') - g(x, z))  \exp(g(x,z'))dx}{\int \exp(g(x, z')) dx}\\
& \leq 2 C \|z - z'\|^\gamma_1,
\end{align*}
where we used Jensen's inequality in the next to last inequality. Reversing the roles of $z $ and $z'$ we complete the proof.

\end{proof}

\end{document}